\documentclass[12pt]{article}

\usepackage{amsfonts,amsmath,amsthm,amssymb,amscd}
\usepackage[dvips]{graphicx}
\usepackage{mathrsfs}

\theoremstyle{definition}
\newtheorem{theorem}{Theorem}
\newtheorem{lemma}[theorem]{Lemma}
\newtheorem{proposition}[theorem]{Proposition}

\newtheorem{definition}[theorem]{Definition}
\numberwithin{equation}{section}
\numberwithin{theorem}{section}

\newcommand{\Ker}{\operatorname{Ker}}
\newcommand{\sign}{\operatorname{sign}}

\begin{document}

\begin{center}
{\bf{\Large Analogue of the theta group $\Gamma_{\theta}$ }}
\end{center}

\begin{center}
By Kazuhide Matsuda
\end{center}

\begin{center}
Faculty of Fundamental Science, \\
National Institute of Technology (KOSEN), Niihama College,\\
7-1 Yagumo-chou, Niihama, Ehime, Japan, 792-8580. \\
E-mail: ka.matsuda@niihama-nct.ac.jp  \\
Fax: 81-0897-37-7809 
\end{center}

\noindent
{\bf Abstract}
In this paper, we introduce higher-level versions of the theta group $\Gamma_{\theta}.$ 
In particular, we treat the theta groups of levels $3$ and $4$, $\Gamma_{\theta,3}$ and $\Gamma_{\theta,4},$  
and 
we prove that 
$\displaystyle F(\tau)=\eta \left(\frac{\tau-1}{3} \right) \eta\left(\frac{\tau+1}{3} \right)$ 
and 
$\displaystyle G(\tau)=\eta \left(\frac{\tau-1}{4} \right) \eta\left(\frac{\tau+1}{4} \right)$ 
are modular forms on $\Gamma_{\theta,3}$ and $\Gamma_{\theta,4},$ respectively. 
Moreover,  
we compute their multiplier systems, $\nu_{F}$ and $\nu_{G}$.  

\section{Introduction}
\label{intro}
Throughout this paper, 
we adopt the definitions and the notations of modular groups and modular forms by Knopp \cite{Knopp}. 
We define 
\begin{equation*}
\Gamma(1)=SL(2,\mathbb{Z}) \,\,
\mathrm{and} \,\,
\Gamma_{\theta}=
\left\{
\begin{pmatrix}
a & b \\
c & d
\end{pmatrix}
\in
\Gamma(1) \,\Big| \,
a\equiv d  \bmod{2} \,\,
\mathrm{and} \,\,
b \equiv c  \bmod{2} 
\right\}.
\end{equation*}
In particular, 
$\Gamma_{\theta}$ is called the {\it theta group}. 
\par
In this paper, we  propose the theta group of level $N$ as follows: 
\begin{equation*}
\Gamma_{\theta,N}
=
\left\{
\begin{pmatrix}
a & b \\
c & d
\end{pmatrix}
\in
\Gamma(1) \,\Big| \,
a\equiv d  \bmod{N} \,\,
\mathrm{and} \,\,
b \equiv -c  \bmod{N} 
\right\}, \,\,
(N=1,2,3,\ldots).
\end{equation*}
Ogg \cite{Ogg} treated the modular group $G(\lambda)$ generated by 
$\tau\to \tau+\lambda$ and $\tau \to -1/\tau.$ 
\par
The {\it upper half plane} $\mathscr{H}$ is defined by 
$
\mathscr{H}=
\{
\tau\in\mathbb{C} \, | \,  \Im \tau>0
\}. 
$
Moreover, we set $q=\exp(2\pi i  \tau)$ and define 
\begin{equation*}
\displaystyle
\eta(\tau)=q^{\frac{1}{24}} \prod_{n=1}^{\infty} (1-q^n) 
\,\,\mathrm{and} \,\,
\theta(\tau)=\sum_{n=-\infty}^{\infty}  e^{\pi i n^2 \tau}.  
\end{equation*}
\par
It is well known that 
$\eta(\tau)$ is a modular form of weight $\displaystyle 1/2$ on $\Gamma(1)$ and 
$\theta(\tau)$ is a modular form of weight $\displaystyle 1/2$ on $\Gamma_{\theta}.$ 
The multiplier systems of $\eta(\tau)$ and $\theta(\tau)$ $\nu_{\eta}, \nu_{\theta},$ are given in Knopp \cite[p.51]{Knopp}
In \cite{Matsuda}, 
from the multiplier systems, $\nu_{\eta}$ and $ \nu_{\theta},$ 
we obtained many examples of modular groups. 
\par
This paper 
aims to 
construct modular forms on $\Gamma_{\theta,3}$ and $\Gamma_{\theta,4}$ and compute their multiplier systems. 
Section \ref{sec:eta} 
describes the multiplier system $\nu_{\eta}$ of $\eta(\tau)$. 
Section \ref{sec:level3} 
treats
$\displaystyle F(\tau)=\eta \left(\frac{\tau-1}{3} \right) \eta\left(\frac{\tau+1}{3} \right)$ 
and its multiplier system $\nu_{F}.$ 
Section \ref{sec:level4} 
deals with 
$\displaystyle G(\tau)=\eta \left(\frac{\tau-1}{4} \right) \eta\left(\frac{\tau+1}{4} \right)$ 
and its multiplier system $\nu_{G}.$ 
Moreover, we discover many examples of modular groups by considering the kernels of $\nu_{F}$ and $\nu_{G}.$ 
Section \ref{sec:coset} 
considers 
coset decompositions of $\Gamma(1)$ modulo $\Gamma_{\theta,3}$ and $\Gamma_{\theta,4},$ 
which yield parabolic points and elleptic elements of $\Gamma_{\theta,3}$ and $\Gamma_{\theta,4},$


\section{Multiplier system of $\eta(\tau)$}
\label{sec:eta}

Knopp \cite[pp.51]{Knopp} proved that for each 
$
M
=
\begin{pmatrix}
a & b \\
c & d
\end{pmatrix}
\in
\Gamma(1),
$
\begin{equation}
\label{eqn-eta-multiplier}
\nu_{\eta}(M)=
\begin{cases}
\displaystyle
\left(
\frac{d}{c}
\right)^{*}   
\exp
\left\{
\frac{\pi i}{12}
[
(a+d)c-bd(c^2-1)-3c
]
\right\} 
&
\text{if $c$ is odd,}
\vspace{2mm}  \\
\displaystyle
\left(
\frac{c}{d}
\right)_{*}
\exp
\left\{
\frac{\pi i}{12}
[
(a+d)c-bd(c^2-1)+3d-3-3cd
]
\right\} 
&
\text{if $c$ is even,}
\end{cases}
\end{equation}
where for $c,d\in \mathbb{Z}$ with $(c,d)=1$ and $d\equiv 1 \bmod{2},$ 
\begin{equation*}
\left(
\frac{c}{d}
\right)^{*}=
\left(
\frac{c}{ |d|}
\right)
\mathrm{and} 
\left(
\frac{c}{d}
\right)_{*}
=
\left(
\frac{c}{ |d|}
\right)
(-1)^{
\frac{\sign c-1}{2}
\frac{\sign d-1}{2}
}. 
\end{equation*}

\section{Modular form on $\Gamma_{\theta,3}$ }
\label{sec:level3}

We note that  
\begin{equation*}
\Gamma_{\theta,3}
=
\left\{
\begin{pmatrix}
a & b \\
c & d
\end{pmatrix}
\in
\Gamma(1) \,\Big| \,
\begin{pmatrix}
a & b \\
c & d
\end{pmatrix}
\equiv 
\pm
\begin{pmatrix}
1 & 0 \\
0 & 1
\end{pmatrix}, 
\pm
\begin{pmatrix}
0 & -1 \\
1 & 0
\end{pmatrix} 
\bmod{3}
\right\}.
\end{equation*}

\begin{proposition}
\label{prop:character-level3}
{\it
For every $\tau\in\mathscr{H},$
set 
\begin{equation*}
F(\tau)= \eta \left(\frac{\tau-1}{3} \right) \eta\left(\frac{\tau+1}{3} \right).
\end{equation*}
The multiplier system $\nu_{F}$ is a character. 
}
\end{proposition}

\begin{proof}
Set 
$
M=
\begin{pmatrix}
a & b \\
c & d
\end{pmatrix}
\in 
\Gamma_{\theta,3}. 
$ 
We note that $\nu_{F}$ is a character if the weight $r$ is a rational integer.  
\par
We first suppose that 
$
M
\equiv \pm
\begin{pmatrix}
1 & 0 \\
0 & 1
\end{pmatrix} 
\bmod{3},$ 
which implies that  
\begin{align*}
F(M\tau)=&
\eta \left(\frac{ M \tau-1}{3} \right) \eta\left(\frac{ M\tau+1}{3} \right)
=
\eta \left( M_1 \left( \frac{  \tau-1}{3} \right)  \right) \eta\left( M_2  \left( \frac{ \tau+1}{3} \right)  \right) \\
=&
\nu_{\eta}(M_1) 
\left(
3c \cdot 
\frac{\tau-1}{3}
+(d+c)
\right)^{\frac12}
\eta \left(\frac{\tau-1}{3} \right)
\nu_{\eta} (M_2)
\left(
3c \cdot 
\frac{\tau+1}{3}
+(d-c)
\right)^{\frac12}
\eta\left(\frac{\tau+1}{3} \right)  \\
=&
\nu_{\eta}(M_1) \nu_{\eta} (M_2)
(c\tau+d)F(\tau),
\end{align*}
where 
\begin{equation*}
M_1
=
\begin{pmatrix}
a-c & \frac{ b-d+a-c }{3} \\
3c & d+c
\end{pmatrix}, 
M_2
=
\begin{pmatrix}
a+c & \frac{ b+d-a-c }{3} \\
3c & d-c
\end{pmatrix}. 
\end{equation*}
\par
We next assume that 
$
M
\equiv \pm
\begin{pmatrix}
0 & -1 \\
1 & 0
\end{pmatrix} 
\bmod{3},$ 
which implies that  
\begin{align*}
F(M\tau)=&
\eta \left(\frac{ M \tau-1}{3} \right) \eta\left(\frac{ M\tau+1}{3} \right)
=
\eta \left( M_1 \left( \frac{  \tau+1}{3} \right)  \right) \eta\left( M_2  \left( \frac{ \tau-1}{3} \right)  \right) \\
=&
\nu_{\eta}(M_1) 
\left(
3c \cdot 
\frac{\tau+1}{3}
+(d-c)
\right)^{\frac12}
\eta \left(\frac{\tau+1}{3} \right) 
\nu_{\eta} (M_2)
\left(
3c \cdot 
\frac{\tau-1}{3}
+(d+c)
\right)^{\frac12}
\eta\left(\frac{\tau-1}{3} \right)  \\
=&
\nu_{\eta}(M_1) \nu_{\eta} (M_2)
(c\tau+d)F(\tau),
\end{align*}
where 
\begin{equation*}
M_1
=
\begin{pmatrix}
a-c & \frac{ b-d-a+c }{3} \\
3c & d-c
\end{pmatrix}, 
M_2
=
\begin{pmatrix}
a+c & \frac{ b+d+a+c }{3} \\
3c & d+c
\end{pmatrix}. 
\end{equation*}
\end{proof}

\subsection{Case where 
$
M
=
\begin{pmatrix}
a & b \\
c & d
\end{pmatrix}
\equiv \pm
\begin{pmatrix}
1 & 0 \\
0 & 1
\end{pmatrix} 
\bmod{3}
$ 
and 
$c$ is odd
}

\begin{theorem}
\label{thm:M=pmI-c-odd}
{\it
Suppose that 
$
M=
\begin{pmatrix}
a & b \\
c & d
\end{pmatrix}
\in 
\Gamma_{\theta,3}
$ 
and 
$M\equiv 
\pm
\begin{pmatrix}
1 & 0 \\
0 & 1
\end{pmatrix} 
\bmod{3}
$ 
and 
$c$ is odd. 
Then, 
the multiplier system 
$\nu_{F}$ of 
$
\displaystyle 
F=
\eta \left(\frac{\tau-1}{3} \right) \eta\left(\frac{\tau+1}{3} \right)
$ 
is given by 
\begin{equation*}
\nu_{F}
(M)
=
\exp
\left\{
\frac{\pi i}{6}
\left[
3c+\frac13(a+d)c+\frac43bd
\right]
\right\}. 
\end{equation*}
}
\end{theorem}

\begin{proof}
We note that 
\begin{align*}
F(M\tau)=&
\nu_{\eta}(M_1) \nu_{\eta} (M_2)
(c\tau+d)F(\tau),
\end{align*}
where 
\begin{equation*}
M_1
=
\begin{pmatrix}
a-c & \frac{ b-d+a-c }{3} \\
3c & d+c
\end{pmatrix}, 
M_2
=
\begin{pmatrix}
a+c & \frac{ b+d-a-c }{3} \\
3c & d-c
\end{pmatrix}. 
\end{equation*}
Equation (\ref{eqn-eta-multiplier})
yields 
\begin{align*}
\nu_{\eta}(M_1)=&
\left(
\frac{d+c}{3c}
\right)^{*}  
\exp
\Big\{
\frac{\pi i}{12}
\Big[
-3 a c^3-3 a c^2 d-3 b c^3-3 b c^2 d +3 c^4+6 c^3 d+3 c^2 d^2 -9 c  \\
&
\hspace{40mm}
+\frac{10 a c}{3}+\frac{a d}{3}
+\frac{b c}{3}+\frac{b d}{3}
-\frac{c^2}{3}+\frac{7 c d}{3}-\frac{d^2}{3}
\Big]
\Big\},
\end{align*}
and 
\begin{align*}
\nu_{\eta}(M_2)=&
\left(
\frac{d-c}{3c}
\right)^{*}  
\exp 
\Big\{
\frac{\pi i}{12}
\Big[
-3 a c^3+3 a c^2 d+3 b c^3-3 b c^2 d-3 c^4+6 c^3 d-3 c^2 d^2-9 c  \\
&
\hspace{40mm}
+\frac{10 a c}{3}-\frac{a d}{3}-\frac{b c}{3}+\frac{b d}{3}
+\frac{c^2}{3}+\frac{7 c d}{3}+\frac{d^2}{3}
\Big]
\Big\},
\end{align*}
which implies that 
\begin{equation*}
\nu_{F}
=
\nu_{\eta}(M_1) \nu_{\eta} (M_2)
=
\left(
\frac{d+c}{3c}
\right)^{*}  
\left(
\frac{d-c}{3c}
\right)^{*}  
\exp
\left\{
\frac{\pi i}{12}
E
\right\}, 
\end{equation*}
where 
\begin{equation*}
E=
-18 c
-6 a c^3+\frac{20 a c}{3}-6 b c^2 d+\frac{2 b d}{3}+12 c^3 d+\frac{14 c d}{3}. 
\end{equation*}
\par
Since $c\equiv 0 \bmod{3},$ it follows that 
\begin{equation*}
\left(
\frac{d+c}{3c}
\right)^{*}  
\left(
\frac{d-c}{3c}
\right)^{*}  
=
\left(
\frac{d+c}{3 |c|}
\right) 
\left(
\frac{d-c}{3 |c|}
\right)=
\left(
\frac{d^2-c^2}{3 |c|}
\right)=
 \left(
\frac{d^2}{3 |c|}
\right)=1. 
\end{equation*} 
\par
Since $c$ is odd, 
it follows that 
$c^2\equiv 1 \bmod{8},$ 
which implies that 
\begin{align*}
E=&-18 c
-6 a c^3+\frac{20 a c}{3}-6 b c^2 d+\frac{2 b d}{3}+12 c^3 d+\frac{14 c d}{3}  & \\
\equiv & \,
6c-6ac+\frac{20}{3} ac-6bd+\frac23bd+12cd+\frac{14}{3}cd &  &\bmod{24}  \\
\equiv &\, 
6c+\frac{2}{3}ac
-\frac{16}{3}bd
-12cd
+\frac{14}{3}cd  &  &\bmod{24}  \\
\equiv &
\, 
6c+\frac{2}{3}ac
-\frac{16}{3}bd-\frac{22}{3}cd &  &\bmod{24}  \\
\equiv &
6c+\frac{2}{3}ac
+\frac{8}{3}bd+\frac{2}{3}cd
\equiv
2
\left[
3c+\frac13(a+d)c+\frac43bd
\right]   &   &\bmod{24},
\end{align*}
which proves the theorem. 
\end{proof}

\subsection{Case where 
$
M
=
\begin{pmatrix}
a & b \\
c & d
\end{pmatrix}
\equiv \pm
\begin{pmatrix}
0 & -1 \\
1 & 0
\end{pmatrix} 
\bmod{3}
$ 
and 
$c$ is odd
}

\begin{theorem}
\label{thm:M=pmT-c-odd}
{\it
Suppose that 
$
M=
\begin{pmatrix}
a & b \\
c & d
\end{pmatrix}
\in 
\Gamma_{\theta,3}
$ 
and 
$M\equiv 
\pm
\begin{pmatrix}
0 & -1 \\
1 & 0
\end{pmatrix} 
\bmod{3}
$ 
and 
$c$ is odd. 
Then, 
the multiplier system 
$\nu_{F}$ of 
$
\displaystyle 
F=
\eta \left(\frac{\tau-1}{3} \right) \eta\left(\frac{\tau+1}{3} \right)
$ 
is given by 
\begin{equation*}
\nu_{F}
(M)
=
\exp
\left\{
\frac{\pi i}{6}
\left[
3(c+2)+\frac13(a+d)c
\right]
\right\}. 
\end{equation*}
In particular, we have $\nu_{F}(T)=-i,$ where 
$
T=
\begin{pmatrix}
0 & -1 \\
1 & 0
\end{pmatrix}.  
$
}
\end{theorem}

\begin{proof}
We note that 
\begin{align*}
F(M\tau)=&
\nu_{\eta}(M_1) \nu_{\eta} (M_2)
(c\tau+d)F(\tau),
\end{align*}
where 
\begin{equation*}
M_1
=
\begin{pmatrix}
a-c & \frac{ b-d-a+c }{3} \\
3c & d-c
\end{pmatrix}, 
M_2
=
\begin{pmatrix}
a+c & \frac{ b+d+a+c }{3} \\
3c & d+c
\end{pmatrix}. 
\end{equation*}
Equation (\ref{eqn-eta-multiplier})
yields 
\begin{align*}
\nu_{\eta}(M_1)=&
\left(
\frac{d-c}{3c}
\right)^{*}  
\exp
\Big\{
\frac{\pi i}{12}
\Big[
-3 a c^3+3 a c^2 d+3 b c^3-3 b c^2 d+3 c^4-6 c^3 d+3 c^2 d^2 -9 c\\ 
&
\hspace{40mm}
+\frac{10 a c}{3}-\frac{a d}{3}-\frac{b c}{3}+\frac{b d}{3}-\frac{19 c^2}{3}+\frac{11 c d}{3}-\frac{d^2}{3}
\Big]
\Big\},
\end{align*}
and 
\begin{align*}
\nu_{\eta}(M_2)=&
\left(
\frac{d+c}{3c}
\right)^{*}  
\exp 
\Big\{
\frac{\pi i}{12}
\Big[
-3 a c^3-3 a c^2 d-3 b c^3-3 b c^2 d-3 c^4-6 c^3 d-3 c^2 d^2-9 c \\
&\hspace{40mm}
+\frac{10 a c}{3}+\frac{a d}{3}
+\frac{b c}{3}+\frac{b d}{3}+\frac{19 c^2}{3}+\frac{11 c d}{3}+\frac{d^2}{3}
\Big]
\Big\},
\end{align*}
which implies that 
\begin{equation*}
\nu_{F}
=
\nu_{\eta}(M_1) \nu_{\eta} (M_2)
=
\left(
\frac{d-c}{3c}
\right)^{*}  
\left(
\frac{d+c}{3c}
\right)^{*}  
\exp
\left\{
\frac{\pi i}{12}
E
\right\}, 
\end{equation*}
where 
\begin{equation*}
E=
-18 c + \frac{20}{3} a c - 6 a c^3 + \frac23 b d + \frac{22}{3} c d - 6 b c^2 d - 
 12 c^3 d. 
\end{equation*}
\par
As $c\equiv \pm1 \bmod{3}$ and $d\equiv0 \bmod{3},$ it follows that 
\begin{align*}
\left(
\frac{d-c}{3c}
\right)^{*}  
\left(
\frac{d+c}{3c}
\right)^{*}  
=&
\left(
\frac{d-c}{3 |c|}
\right) 
\left(
\frac{d+c}{3 |c|}
\right)=
\left(
\frac{d-c}{3}
\right)
\left(
\frac{d-c}{ |c|}
\right) 
\left(
\frac{d+c}{3}
\right)
\left(
\frac{d+c}{ |c|}
\right)  \\
=&\left(
\frac{d^2-c^2}{3}
\right)
\left(
\frac{d^2-c^2}{ |c|}
\right) 
=
\left(
\frac{-1}{3}
\right)
\left(
\frac{d^2}{ |c|}
\right) 
=-1. 
\end{align*} 
\par
Because $c$ is odd, 
it follows that 
$c^2\equiv 1 \bmod{8},$ 
which implies that 
\begin{align*}
E=&-18 c + \frac{20}{3} a c - 6 a c^3 + \frac23 b d + \frac{22}{3} c d - 6 b c^2 d - 
 12 c^3 d & \\
\equiv & \,
6 c + \frac{20}{3} a c - 6 a c + \frac23 b d + \frac{22}{3} c d - 6 b d - 
 12 c d  &  &\bmod{24}  \\
\equiv &\, 
6 c + \frac{2}{3} a c - \frac{16}3 b d - \frac{14}{3} c d &  &\bmod{24}  \\
\equiv &
\, 
6c+\frac{2}{3}ac
+\frac{8}{3}bd+\frac{10}{3}cd  &  &\bmod{24}  \\
\equiv &\,
6c+\frac{2}{3}(a+d)c
+\frac{8}{3}(b+c)d
\equiv
2
\left[
3c+\frac13(a+d)c
\right]    &  &\bmod{24},
\end{align*}
which proves the theorem. 
\end{proof}

\subsection{Case where 
$
M
=
\begin{pmatrix}
a & b \\
c & d
\end{pmatrix}
\equiv \pm
\begin{pmatrix}
1 & 0 \\
0 & 1
\end{pmatrix} 
\bmod{3}
$ 
and 
$c$ is even
}

\begin{theorem}
\label{thm:M=pmI-c-even}
{\it
Suppose that 
$
M=
\begin{pmatrix}
a & b \\
c & d
\end{pmatrix}
\in 
\Gamma_{\theta,3}
$ 
and 
$M\equiv 
\pm
\begin{pmatrix}
1 & 0 \\
0 & 1
\end{pmatrix} 
\bmod{3}
$ 
and 
$c$ is even. 
Then, 
the multiplier system 
$\nu_{F}$ of 
$
\displaystyle 
F=
\eta \left(\frac{\tau-1}{3} \right) \eta\left(\frac{\tau+1}{3} \right)
$ 
is given by 
\begin{equation*}
\nu_{F}
(M)
=
\exp
\left\{
\frac{\pi i}{6}
\left[
3(d-1)+\frac13(b-c)d
\right]
\right\}. 
\end{equation*}
}
\end{theorem}

\begin{proof}
As $c$ is even, it follows that $d$ is odd. 
By 
$
T
=
\begin{pmatrix}
0 & -1 \\
1 & 0
\end{pmatrix}, 
$ 
we have 
\begin{equation*}
MT
=
\begin{pmatrix}
a & b \\
c & d
\end{pmatrix}
\begin{pmatrix}
0 & -1 \\
1 & 0
\end{pmatrix}
=
\begin{pmatrix}
b & -a \\
d & -c
\end{pmatrix}
\equiv
\pm
\begin{pmatrix}
0 & -1 \\
1 & 0
\end{pmatrix}
\bmod{3}.
\end{equation*}
\par
Proposition \ref{prop:character-level3} and Theorem \ref{thm:M=pmT-c-odd} 
imply that 
\begin{equation*}
\nu_{F}(MT)=\nu_{F}(M)\nu_{F}(T)=\nu_{F}(M)\cdot(-i)
=
\exp
\left\{
\frac{\pi i}{6}
\left[
3d+\frac13(b-c)d+6
\right]
\right\},
\end{equation*}
which proves the theorem. 
\end{proof}

\subsection{Case where 
$
M
=
\begin{pmatrix}
a & b \\
c & d
\end{pmatrix}
\equiv \pm
\begin{pmatrix}
0 & -1 \\
1 & 0
\end{pmatrix} 
\bmod{3}
$ 
and 
$c$ is even
}

\begin{theorem}
\label{thm:M=pmT-c-even}
{\it
Suppose that 
$
M=
\begin{pmatrix}
a & b \\
c & d
\end{pmatrix}
\in 
\Gamma_{\theta,3}
$ 
and 
$M\equiv 
\pm
\begin{pmatrix}
0 & -1 \\
1 & 0
\end{pmatrix} 
\bmod{3}
$ 
and 
$c$ is even. 
Then, 
the multiplier system 
$\nu_{F}$ of 
$
\displaystyle 
F=
\eta \left(\frac{\tau-1}{3} \right) \eta\left(\frac{\tau+1}{3} \right)
$ 
is given by 
\begin{equation*}
\nu_{F}
(M)
=
\exp
\left\{
\frac{\pi i}{6}
\left[
3(d+1)+\frac13(b-c)d+\frac43ac
\right]
\right\}. 
\end{equation*}
}
\end{theorem}

\begin{proof}
As $c$ is even, it follows that $d$ is odd. 
By 
$
T
=
\begin{pmatrix}
0 & -1 \\
1 & 0
\end{pmatrix}, 
$ 
we have 
\begin{equation*}
MT
=
\begin{pmatrix}
a & b \\
c & d
\end{pmatrix}
\begin{pmatrix}
0 & -1 \\
1 & 0
\end{pmatrix}
=
\begin{pmatrix}
b & -a \\
d & -c
\end{pmatrix}
\equiv
\mp
\begin{pmatrix}
1 & 0 \\
0 & 1
\end{pmatrix}
\bmod{3}.
\end{equation*}
\par
Proposition \ref{prop:character-level3} and Theorem \ref{thm:M=pmI-c-odd} 
imply that 
\begin{equation*}
\nu_{F}(MT)=\nu_{F}(M)\nu_{F}(T)=\nu_{F}(M)\cdot(-i)
=
\exp
\left\{
\frac{\pi i}{6}
\left[
3d+\frac13(b-c)d+
\frac43ac
\right]
\right\},
\end{equation*}
which proves the theorem. 
\end{proof}

\subsection{Summary}

\begin{theorem}
\label{summary-level3}
{\it
Set
$
M=
\begin{pmatrix}
a & b \\
c & d
\end{pmatrix}
\in 
\Gamma_{\theta,3}. 
$ 
Then, it follows that 
the multiplier system 
$\nu_{F}$ of 
 $
\displaystyle 
F=
\eta \left(\frac{\tau-1}{3} \right) \eta\left(\frac{\tau+1}{3} \right) 
$ 
is given by 
\begin{equation*}
\nu_{F}(M)=
\begin{cases}
\exp
\left\{
\displaystyle
\frac{\pi i}{6}
\left[
3c+\frac13(a+d)c+\frac43bd
\right]
\right\}   
&
\text
{ 
if
$
M
\equiv 
\pm
\begin{pmatrix}
1 & 0 \\
0 & 1
\end{pmatrix} 
\bmod{3}
$ 
and 
$c$ is odd,
} \vspace{1mm} \\
\exp
\left\{
\displaystyle
\frac{\pi i}{6}
\left[
3(d-1)+\frac13(b-c)d
\right]
\right\}   
&
\text
{ 
if
$
M
\equiv 
\pm
\begin{pmatrix}
1 & 0 \\
0 & 1
\end{pmatrix} 
\bmod{3}
$ 
and 
$c$ is even,
}  \vspace{1mm} \\
\exp
\left\{
\displaystyle
\frac{\pi i}{6}
\left[
3(c+2)+\frac13(a+d)c
\right]
\right\}   
&
\text
{ 
if
$
M
\equiv 
\pm
\begin{pmatrix}
0 & -1 \\
1 & 0
\end{pmatrix} 
\bmod{3}
$ 
and 
$c$ is odd,
}  \vspace{1mm} \\
\exp
\left\{
\displaystyle
\frac{\pi i}{6}
\left[
3(d+1)+\frac13(b-c)d+\frac43ac
\right]
\right\}   
&
\text
{ 
if
$
M
\equiv 
\pm
\begin{pmatrix}
0 & -1 \\
1 & 0
\end{pmatrix} 
\bmod{3}
$ 
and 
$c$ is even.
} 
\end{cases}
\end{equation*}
}
\end{theorem}

\begin{proof}
The theorem follows from 
Theorems 
\ref{thm:M=pmI-c-odd}, 
\ref{thm:M=pmT-c-odd}, 
\ref{thm:M=pmI-c-even}, 
and 
\ref{thm:M=pmT-c-even}. 
\end{proof}

\begin{definition}
For each 
$
M=
\begin{pmatrix}
a & b \\
c & d
\end{pmatrix}
\in 
\Gamma_{\theta,3},
$ 
we define 
\begin{equation*}
f(M)
=
\begin{cases}
\displaystyle
3c+\frac13(a+d)c+\frac43bd
&
\text
{ 
if
$
M
\equiv 
\pm
\begin{pmatrix}
1 & 0 \\
0 & 1
\end{pmatrix} 
\bmod{3}
$ 
and 
$c$ is odd,
} \vspace{1mm} \\
\displaystyle
3(d-1)+\frac13(b-c)d
&
\text
{ 
if
$
M
\equiv 
\pm
\begin{pmatrix}
1 & 0 \\
0 & 1
\end{pmatrix}
\bmod{3}
$ 
and 
$c$ is even,
}  \vspace{1mm} \\
\displaystyle
3(c+2)+\frac13(a+d)c
&
\text
{ 
if
$
M
\equiv 
\pm
\begin{pmatrix}
0 & -1 \\
1 & 0
\end{pmatrix}
\bmod{3}
$ 
and 
$c$ is odd,
}  \vspace{1mm} \\
\displaystyle
3(d+1)+\frac13(b-c)d+\frac43ac
&
\text
{ 
if
$
M
\equiv 
\pm
\begin{pmatrix}
0 & -1 \\
1 & 0
\end{pmatrix} 
\bmod{3}
$ 
and 
$c$ is even.
} 
\end{cases}
\end{equation*}
\end{definition}

From the definition, 
it follows that, for each $k\in\mathbb{Z},$ 
\begin{equation*}
\nu_{F^k}(M)=
\exp \left\{
\frac{\pi i k}{6}
f(M)
\right\}. 
\end{equation*}

\subsection{Case where $k\equiv 0 \bmod{12}$}

\begin{theorem}
{\it
Suppose that 
$k\equiv 0 \bmod{12}$. 
Then, it follows that 
\begin{equation*}
\Ker \nu_{F^{k}}
=\Gamma_{\theta,3}.
\end{equation*}
}
\end{theorem}

\begin{proof}
It is obvious. 
\end{proof}

\subsection{Case where $k\equiv 6 \bmod{12}$}

\begin{theorem}
\label{thmk=6-level3}
{\it
Suppose that 
$k\equiv 6 \bmod{12}$. 
Then, it follows that 
\begin{equation*}
\Ker \nu_{F^{k}}
=
\left\{
\begin{pmatrix}
a & b \\
c & d
\end{pmatrix}
\in
\Gamma_{\theta,3} \,\Big| \, 
\begin{pmatrix}
a & b \\
c & d
\end{pmatrix}
\equiv
\begin{pmatrix}
1 & 0 \\
0 & 1
\end{pmatrix}, 
\begin{pmatrix}
0 & 1 \\
1 & 1
\end{pmatrix},
\begin{pmatrix}
1 & 1 \\
1 & 0
\end{pmatrix}
\bmod{2}
\right\}.
\end{equation*}
Moreover, 
as a coset decomposition of $\Gamma_{\theta,3}$ modulo $\Ker \nu_{F^{k}}$,  
we may choose $S^{n}$ $(n=0,3),$ where 
$
S=
\begin{pmatrix}
1 & 1 \\
0 & 1
\end{pmatrix}. 
$
}
\end{theorem}

\begin{proof}
For each 
$
M
=
\begin{pmatrix}
a & b \\
c & d
\end{pmatrix}
\in
\Gamma_{\theta,3}, 
$ 
we have 
\begin{equation*}
\nu_{F^{k}  }(M)=
\exp
\{\pi i f(M) \}. 
\end{equation*}
\par
If $M \equiv \pm
\begin{pmatrix}
1 & 0 \\
0 & 1
\end{pmatrix}
\bmod{3}
$ 
and $c$ is odd, 
we have 
\begin{align*}
f(M)\equiv 0 \bmod{2}
\Longleftrightarrow& \,
3c+\frac13 (a+d)c+\frac43bd
\equiv c+(a+d)c \equiv 0 \bmod{2} \\
\Longleftrightarrow& \,
1+a+d\equiv 0 \bmod{2} \\
\Longleftrightarrow& \,
(a,d)\equiv (1,0), (0,1) \bmod{2} \\
\Longleftrightarrow& \,
\begin{pmatrix}
a & b \\
c & d
\end{pmatrix}
\equiv 
\begin{pmatrix}
1 & 1 \\
1 & 0
\end{pmatrix},
\begin{pmatrix}
0 & 1 \\
1 & 1
\end{pmatrix}
\bmod{2}. 
\end{align*}
\par
If $M \equiv \pm
\begin{pmatrix}
1 & 0 \\
0 & 1
\end{pmatrix}
\bmod{3}
$ 
and $c$ is even, 
we have 
\begin{align*}
f(M)\equiv 0 \bmod{2}
\Longleftrightarrow& \,
3(d-1)+\frac13(b-c)d
\equiv d-1+(b-c)d \equiv 0 \bmod{2} \\
\Longleftrightarrow& \,
d-1+bd \equiv 0 \bmod{2} \\
\Longleftrightarrow& \,
(b,d)\equiv (0,1) \bmod{2} \\
\Longleftrightarrow& \,
\begin{pmatrix}
a & b \\
c & d
\end{pmatrix}
\equiv 
\begin{pmatrix}
1 & 0 \\
0 & 1
\end{pmatrix}
\bmod{2}. 
\end{align*}
\par
If $M \equiv \pm
\begin{pmatrix}
0 & -1 \\
1 & 0
\end{pmatrix}
\bmod{3}
$ 
and $c$ is odd, 
we have 
\begin{align*}
f(M)\equiv 0 \bmod{2}
\Longleftrightarrow& \,
3(c+2)+\frac13 (a+d)c \equiv c+(a+d)c \equiv 0 \bmod{2} \\
\Longleftrightarrow& \,
1+a+d\equiv 0 \bmod{2} \\
\Longleftrightarrow& \,
(a,d)\equiv (1,0), (0,1) \bmod{2} \\
\Longleftrightarrow& \,
\begin{pmatrix}
a & b \\
c & d
\end{pmatrix}
\equiv 
\begin{pmatrix}
1 & 1 \\
1 & 0
\end{pmatrix},
\begin{pmatrix}
0 & 1 \\
1 & 1
\end{pmatrix}
\bmod{2}. 
\end{align*}
\par
If $M \equiv \pm
\begin{pmatrix}
0 & -1 \\
1 & 0
\end{pmatrix}
\bmod{3}
$ 
and $c$ is even, 
we have 
\begin{align*}
f(M)\equiv 0 \bmod{2}
\Longleftrightarrow& \,
3(d+1)+\frac13 (b-c)d+\frac43ac \equiv d+1+bd \equiv 0 \bmod{2} \\
\Longleftrightarrow& \,
(b,d)\equiv  (0,1) \bmod{2} \\
\Longleftrightarrow& \,
\begin{pmatrix}
a & b \\
c & d
\end{pmatrix}
\equiv 
\begin{pmatrix}
1 & 0 \\
0 & 1
\end{pmatrix}
\bmod{2}. 
\end{align*}
\par
The coset decomposition can be obtained by considering the values of $\nu_{F^{k}  }(M).$
\end{proof}

\subsection{Case where $k\equiv \pm 3 \bmod{12}$}

\begin{lemma}
\label{lem:level3-mod4}
{\it
Set 
$
M=
\begin{pmatrix}
a & b \\
c & d
\end{pmatrix}
\in\Gamma_{\theta,3}.
$
\newline
$\mathrm{(1)}$ If 
$
M
\equiv
\pm
\begin{pmatrix}
1 & 0 \\
0 & 1
\end{pmatrix}
\bmod{3}
$ 
and $c$ is odd, 
we have 
\begin{align}
f(M)\equiv 0 \bmod{4} 
\Longleftrightarrow&
\, a+d\equiv -1 \bmod{4} \notag  \\
\Longleftrightarrow&
M
\equiv
\begin{pmatrix}
0 & -1 \\
1 & -1
\end{pmatrix}, 
\begin{pmatrix}
0 & 1 \\
-1 & -1
\end{pmatrix}, 
\begin{pmatrix}
1 & 1 \\
1 & 2
\end{pmatrix}, 
\begin{pmatrix}
1 & -1 \\
-1 & 2
\end{pmatrix},  \notag \\
&\hspace{10mm}
\begin{pmatrix}
2 & 1 \\
1 & 1
\end{pmatrix}, 
\begin{pmatrix}
2 & -1 \\
-1 & 1
\end{pmatrix}, 
\begin{pmatrix}
-1 & -1 \\
1 & 0
\end{pmatrix}, 
\begin{pmatrix}
-1 & 1 \\
-1 & 0
\end{pmatrix} \bmod{4}.   \label{eqn:level3-mod4-pmI-c-odd}
\end{align}
$\mathrm{(2)}$ If 
$
M
\equiv
\pm
\begin{pmatrix}
1 & 0 \\
0 & 1
\end{pmatrix}
\bmod{3}
$ 
and $c$ is even, we have  
\begin{align}
 \label{eqn:level3-mod4-pmI-c-even}
f(M)\equiv 0 \bmod{4} 
\Longleftrightarrow& 
\, b-c\equiv d-1 \bmod{4} \notag \\
\Longleftrightarrow& \, 
M
\equiv 
\begin{pmatrix}
1 & 0 \\
0 & 1
\end{pmatrix}, 
\begin{pmatrix}
1 & 2 \\
2 & 1
\end{pmatrix}, 
\begin{pmatrix}
-1 & 0 \\
2 & -1
\end{pmatrix}, 
\begin{pmatrix}
-1 & 2 \\
0 & -1
\end{pmatrix} 
\bmod{4}.
\end{align}
$\mathrm{(3)}$ If 
$
M
\equiv
\pm
\begin{pmatrix}
0 & -1 \\
1 & 0
\end{pmatrix}
\bmod{3}
$ 
and $c$ is odd, 
we have 
\begin{align}
f(M)\equiv 0 \bmod{4} 
\Longleftrightarrow&
\, a+d\equiv 1 \bmod{4}  \notag \\
\Longleftrightarrow&
M
\equiv
\begin{pmatrix}
0 & -1 \\
1 & 1
\end{pmatrix}, 
\begin{pmatrix}
0 & 1 \\
-1 & 1
\end{pmatrix}, 
\begin{pmatrix}
1 & -1 \\
1 & 0
\end{pmatrix}, 
\begin{pmatrix}
1 & 1 \\
-1 & 0
\end{pmatrix},  \notag \\
&\hspace{10mm}
\begin{pmatrix}
2 & 1 \\
1 & -1
\end{pmatrix}, 
\begin{pmatrix}
2 & -1 \\
-1 & -1
\end{pmatrix}, 
\begin{pmatrix}
-1 & 1 \\
1 & 2
\end{pmatrix}, 
\begin{pmatrix}
-1 & -1 \\
-1 & 2
\end{pmatrix} \bmod{4}.   \label{eqn:level3-mod4-pmT-c-odd}
\end{align}
$\mathrm{(4)}$ If 
$
M
\equiv
\pm
\begin{pmatrix}
0 & -1 \\
1 & 0
\end{pmatrix}
\bmod{3}
$ 
and $c$ is even, we have  
\begin{align}
\label{eqn:level3-mod4-pmT-c-even}
f(M)\equiv 0 \bmod{4} 
\Longleftrightarrow& \,
b-c\equiv -d-1 \bmod{4} \notag \\
\Longleftrightarrow&
M
\equiv 
\begin{pmatrix}
-1 & 0 \\
0 & -1
\end{pmatrix}, 
\begin{pmatrix}
-1 & 2 \\
2 & -1
\end{pmatrix}, 
\begin{pmatrix}
1 & 0 \\
2 & 1
\end{pmatrix}, 
\begin{pmatrix}
1 & 2 \\
0 & 1
\end{pmatrix} 
\bmod{4}.
\end{align}
}
\end{lemma}

\begin{proof}
We treat cases (1) and (2). The other cases can be proved in the same way. 
\newline
{\bf Case (1).}
From the definition of $f(M),$ it follows that 
\begin{align*}
f(M)\equiv 0 \bmod{4}
\Longleftrightarrow&\, 
3c+\frac13(a+d)c+\frac43bd\equiv 3c-(a+d)c\equiv 0 \bmod{4}  \\
\Longleftrightarrow&\, 
3-a-d\equiv 0 \bmod{4} \\
\Longleftrightarrow&\, 
(a,d)\equiv (0,-1), (1,2), (2,1), (-1,0) \bmod{4}  \\
\Longleftrightarrow&\, 
M
\equiv
\begin{pmatrix}
0 & -1 \\
1 & -1
\end{pmatrix}, 
\begin{pmatrix}
0 & 1 \\
-1 & -1
\end{pmatrix}, 
\begin{pmatrix}
1 & 1 \\
1 & 2
\end{pmatrix}, 
\begin{pmatrix}
1 & -1 \\
-1 & 2
\end{pmatrix},  \\
&\hspace{10mm}
\begin{pmatrix}
2 & 1 \\
1 & 1
\end{pmatrix}, 
\begin{pmatrix}
2 & -1 \\
-1 & 1
\end{pmatrix}, 
\begin{pmatrix}
-1 & -1 \\
1 & 0
\end{pmatrix}, 
\begin{pmatrix}
-1 & 1 \\
-1 & 0
\end{pmatrix} \bmod{4}. 
\end{align*}
{\bf Case (2).}
From the definition of $f(M),$ it follows that 
\begin{align*}
f(M)\equiv 0 \bmod{4}
\Longleftrightarrow&\, 
3(d-1)+\frac13(b-c)d
\equiv 
3(d-1)-(b-c)d \equiv 0 \bmod{4}  \\
\Longleftrightarrow&\, 
(b-c-3)d\equiv 1 \bmod{4}  \\
\Longleftrightarrow&\, 
b-c-3\equiv d.
\end{align*}
If $d\equiv 1 \bmod{4},$ 
we have 
\begin{align*}
b-c-3\equiv 1 \bmod{4}
\Longleftrightarrow&\,
(b,c)\equiv (0,0), (2,2) \bmod{4}   \\
\Longleftrightarrow&\,
M
\equiv
\begin{pmatrix}
1 & 0\\
0 & 1
\end{pmatrix}, 
\begin{pmatrix}
1 & 2\\
2 & 1
\end{pmatrix} \bmod{4}. 
\end{align*}
If $d\equiv -1 \bmod{4},$ 
we have 
\begin{align*}
b-c-3\equiv -1 \bmod{4}
\Longleftrightarrow&\,
(b,c)\equiv (0,2), (2,0) \bmod{4}   \\
\Longleftrightarrow&\,
M
\equiv
\begin{pmatrix}
-1 & 0\\
2 & -1
\end{pmatrix}, 
\begin{pmatrix}
-1 & 2\\
0 & -1
\end{pmatrix} \bmod{4}. 
\end{align*}
\end{proof}

\begin{theorem}
\label{thmk=3-level3}
{\it
Suppose that 
$k\equiv \pm 3 \bmod{12}$. 
Then, it follows that 
\begin{align*}
\Ker \nu_{F^{k}}
=
\Big\{
M=
\begin{pmatrix}
a & b \\
c & d
\end{pmatrix}
\in
\Gamma_{\theta,3} \,\Big| \, 
  &a+d \equiv& -1 &  &\bmod{4}  &  &(c\equiv 0 \bmod{3}, c\equiv 1 \bmod{2}),  \\
  &b-c \equiv& d-1&  &\bmod{4}  & &(c\equiv 0 \bmod{3}, c \equiv 0 \bmod{2}), \\
  &a+d \equiv& 1 &  &\bmod{4}  &  &(c \not\equiv 0 \bmod{3}, c \equiv 1 \bmod{2}),   \\
  &b-c \equiv& -d-1 &  &\bmod{4} & &(c\not\equiv 0 \bmod{3}, c\equiv 0 \bmod{2})
\Big\}.
\end{align*}
Moreover, 
as a coset decomposition of $\Gamma_{\theta,3}$ modulo $\Ker \nu_{F^{k}}$,  
we may choose $S^{n}$ $(n=0,3,6,9),$ where 
$
S=
\begin{pmatrix}
1 & 1 \\
0 & 1
\end{pmatrix}. 
$
}
\end{theorem}

\begin{proof}
For each $M\in\Gamma_{\theta,3},$ 
we have 
\begin{equation*}
\nu_{F^k}(M)=
\exp
\left\{
\pm
\frac{\pi i}{2}
f(M)
\right\},
\end{equation*}
which proves the theorem. 
\par
The coset decomposition can be obtained by considering the values of $\nu_{F^{k}  }(M).$
\end{proof}

\subsection{Case where $k\equiv \pm 4 \bmod{12}$}

\begin{lemma}
\label{lem:level3-mod3}
{\it
Set 
$
M=
\begin{pmatrix}
a & b \\
c & d
\end{pmatrix}
\in\Gamma_{\theta,3}.
$
\newline
$\mathrm{(1)}$ If 
$
M
\equiv
\pm
\begin{pmatrix}
1 & 0 \\
0 & 1
\end{pmatrix}
\bmod{3}
$ 
and $c$ is odd, 
we have 
\begin{align*}
f(M)\equiv 0 \bmod{3} 
\Longleftrightarrow
\frac{b}{3} \equiv \frac{c}{3} \bmod{3}
\Longleftrightarrow
M
\equiv& 
\pm
\begin{pmatrix}
1 & 0 \\
0 & 1
\end{pmatrix}, 
\pm
\begin{pmatrix}
1 & 3 \\
3 & 1
\end{pmatrix}, 
\pm
\begin{pmatrix}
1 & -3 \\
-3 & 1
\end{pmatrix},  \notag \\
&
\pm
\begin{pmatrix}
4 & 0 \\
0 & -2
\end{pmatrix},  
\pm
\begin{pmatrix}
4 & 3 \\
3 & -2
\end{pmatrix}, 
\pm
\begin{pmatrix}
4 & -3 \\
-3 & -2
\end{pmatrix}, \notag \\
&
\pm
\begin{pmatrix}
-2 & 0 \\
0 & 4
\end{pmatrix}, 
\pm
\begin{pmatrix}
-2 & 3 \\
3 & 4
\end{pmatrix}, 
\pm
\begin{pmatrix}
-2 & -3 \\
-3 & 4
\end{pmatrix} 
\bmod{9}.   
\end{align*}
$\mathrm{(2)}$ If 
$
M
\equiv
\pm
\begin{pmatrix}
1 & 0 \\
0 & 1
\end{pmatrix}
\bmod{3}
$ 
and $c$ is even, 
we have 
\begin{align*}
f(M)\equiv 0 \bmod{3} 
\Longleftrightarrow
\frac{b}{3} \equiv \frac{c}{3} \bmod{3}
\Longleftrightarrow
M
\equiv& 
\pm
\begin{pmatrix}
1 & 0 \\
0 & 1
\end{pmatrix}, 
\pm
\begin{pmatrix}
1 & 3 \\
3 & 1
\end{pmatrix}, 
\pm
\begin{pmatrix}
1 & -3 \\
-3 & 1
\end{pmatrix},  \notag \\
&
\pm
\begin{pmatrix}
4 & 0 \\
0 & -2
\end{pmatrix},  
\pm
\begin{pmatrix}
4 & 3 \\
3 & -2
\end{pmatrix}, 
\pm
\begin{pmatrix}
4 & -3 \\
-3 & -2
\end{pmatrix}, \notag \\
&
\pm
\begin{pmatrix}
-2 & 0 \\
0 & 4
\end{pmatrix}, 
\pm
\begin{pmatrix}
-2 & 3 \\
3 & 4
\end{pmatrix}, 
\pm
\begin{pmatrix}
-2 & -3 \\
-3 & 4
\end{pmatrix} 
\bmod{9}.   
\end{align*}
$\mathrm{(3)}$ If 
$
M
\equiv
\pm
\begin{pmatrix}
0 & -1 \\
1 & 0
\end{pmatrix}
\bmod{3}
$ 
and $c$ is odd, 
we have 
\begin{align*}
f(M)\equiv 0 \bmod{3} 
\Longleftrightarrow
\frac{a}{3} \equiv -\frac{d}{3} \bmod{3}
\Longleftrightarrow
M
\equiv& 
\pm
\begin{pmatrix}
0 & -1 \\
1 & 0
\end{pmatrix}, 
\pm
\begin{pmatrix}
0 & 2 \\
4 & 0
\end{pmatrix}, 
\pm
\begin{pmatrix}
0 & -4 \\
-2 & 0
\end{pmatrix},  \notag \\
&
\pm
\begin{pmatrix}
3 & -1 \\
1 & -3
\end{pmatrix},  
\pm
\begin{pmatrix}
3 & 2 \\
4 & -3
\end{pmatrix}, 
\pm
\begin{pmatrix}
3 & -4 \\
-2 & -3
\end{pmatrix}, \notag \\
&
\pm
\begin{pmatrix}
-3 & -1 \\
1 & 3
\end{pmatrix}, 
\pm
\begin{pmatrix}
-3 & 2 \\
4 & 3
\end{pmatrix},  
\pm
\begin{pmatrix}
-3 & -4 \\
-2 & 3
\end{pmatrix} 
\bmod{9}.  
\end{align*}
$\mathrm{(4)}$ If 
$
M
\equiv
\pm
\begin{pmatrix}
0 & -1 \\
1 & 0
\end{pmatrix}
\bmod{3}
$ 
and $c$ is even, 
we have 
\begin{align*}
f(M)\equiv 0 \bmod{3} 
\Longleftrightarrow
\frac{a}{3} \equiv -\frac{d}{3} \bmod{3}
\Longleftrightarrow
M
\equiv& 
\pm
\begin{pmatrix}
0 & -1 \\
1 & 0
\end{pmatrix}, 
\pm
\begin{pmatrix}
0 & 2 \\
4 & 0
\end{pmatrix}, 
\pm
\begin{pmatrix}
0 & -4 \\
-2 & 0
\end{pmatrix},  \notag \\
&
\pm
\begin{pmatrix}
3 & -1 \\
1 & -3
\end{pmatrix},  
\pm
\begin{pmatrix}
3 & 2 \\
4 & -3
\end{pmatrix}, 
\pm
\begin{pmatrix}
3 & -4 \\
-2 & -3
\end{pmatrix}, \notag \\
&
\pm
\begin{pmatrix}
-3 & -1 \\
1 & 3
\end{pmatrix}, 
\pm
\begin{pmatrix}
-3 & 2 \\
4 & 3
\end{pmatrix},  
\pm
\begin{pmatrix}
-3 & -4 \\
-2 & 3
\end{pmatrix} 
\bmod{9}.   
\end{align*}
}
\end{lemma}

\begin{proof}
We treat case (1). The other cases can be proved in the same way. 
\par
From the definition of $f(M),$ 
it follows that 
\begin{align*}
f(M) \equiv 0 \bmod{3} 
\Longleftrightarrow &\,
3c+\frac13(a+d)c+\frac43bd 
\equiv
\frac13(a+d)c+\frac13bd
\equiv 0 \bmod{3}  \\
\Longleftrightarrow &\,
\frac13(a+d)c\equiv -\frac13bd  \bmod{3}.
\end{align*}
\par
If 
$M\equiv 
\begin{pmatrix}
1 & 0 \\
0 & 1
\end{pmatrix} 
\bmod{3},$ 
we have 
\begin{align*}
\frac13(a+d)c\equiv -\frac13bd  \bmod{3} 
\Longleftrightarrow &\,
\frac23 c\equiv -\frac13 b \bmod{3} \\
\Longleftrightarrow &\,
\frac{b}{3} \equiv \frac{c}{3} \equiv 0, \pm1 \bmod{3} \\
\Longleftrightarrow &\,
(b,c)\equiv (0,0), (3,3), (-3,-3) \bmod{9}.
\end{align*}
Because $ad-bc=1,$ it follows that $ad\equiv 1 \bmod{9},$ 
which implies that 
\begin{equation*}
(a,d)\equiv (1,1), (4,-2), (-2,4) \bmod{9}.
\end{equation*}
\par
If 
$M\equiv 
\begin{pmatrix}
-1 & 0 \\
0 & -1
\end{pmatrix} 
\bmod{3},$ 
we have 
\begin{align*}
\frac13(a+d)c\equiv -\frac13bd  \bmod{3} 
\Longleftrightarrow &\,
-\frac23 c\equiv \frac13 b \bmod{3} \\
\Longleftrightarrow &\,
\frac{b}{3} \equiv \frac{c}{3} \equiv 0, \pm1 \bmod{3} \\
\Longleftrightarrow &\,
(b,c)\equiv (0,0), (3,3), (-3,-3) \bmod{9}.
\end{align*}
Because $ad-bc=1,$ it follows that $ad\equiv 1 \bmod{9},$ 
which implies that 
\begin{equation*}
(a,d)\equiv (-1,-1), (2,-4), (-4,2) \bmod{9}.
\end{equation*}
\end{proof}

\begin{theorem}
\label{thmk=4-level3}
{\it
Suppose that 
$k\equiv \pm 4 \bmod{12}$. 
Then, it follows that 
\begin{align*}
\Ker \nu_{F^{k}}
=&
\left\{
\begin{pmatrix}
a & b \\
c & d
\end{pmatrix} 
\in\Gamma_{\theta,3} \,\,\Big| \,\,
\frac{b}{3} \equiv\frac{c}{3} \bmod{3}  \, (b,c\in3\mathbb{Z}) \,\,
\mathrm{or} \,\,
\frac{a}{3}\equiv-\frac{d}{3} \bmod{3} \, (a,d\in3\mathbb{Z})
\right\} \\
=&
\Big\{
M
\in
\Gamma_{\theta,3} \,\Big| \, 
M \equiv 
\pm
\begin{pmatrix}
1 & 0 \\
0 & 1
\end{pmatrix}, 
\pm
\begin{pmatrix}
1 & 3 \\
3 & 1
\end{pmatrix}, 
\pm
\begin{pmatrix}
1 & -3 \\
-3 & 1
\end{pmatrix},   \\
&\hspace{31mm}
\pm
\begin{pmatrix}
4 & 0 \\
0 & -2
\end{pmatrix},  
\pm
\begin{pmatrix}
4 & 3 \\
3 & -2
\end{pmatrix}, 
\pm
\begin{pmatrix}
4 & -3 \\
-3 & -2
\end{pmatrix}, \notag \\
&\hspace{31mm}
\pm
\begin{pmatrix}
-2 & 0 \\
0 & 4
\end{pmatrix}, 
\pm
\begin{pmatrix}
-2 & 3 \\
3 & 4
\end{pmatrix}, 
\pm
\begin{pmatrix}
-2 & -3 \\
-3 & 4
\end{pmatrix},   \\
&\hspace{31mm}
\pm
\begin{pmatrix}
0 & -1 \\
1 & 0
\end{pmatrix}, 
\pm
\begin{pmatrix}
0 & 2 \\
4 & 0
\end{pmatrix}, 
\pm
\begin{pmatrix}
0 & -4 \\
-2 & 0
\end{pmatrix},  \notag \\
&\hspace{31mm}
\pm
\begin{pmatrix}
3 & -1 \\
1 & -3
\end{pmatrix},  
\pm
\begin{pmatrix}
3 & 2 \\
4 & -3
\end{pmatrix}, 
\pm
\begin{pmatrix}
3 & -4 \\
-2 & -3
\end{pmatrix}, \notag \\
&\hspace{31mm}
\pm
\begin{pmatrix}
-3 & -1 \\
1 & 3
\end{pmatrix}, 
\pm
\begin{pmatrix}
-3 & 2 \\
4 & 3
\end{pmatrix},  
\pm
\begin{pmatrix}
-3 & -4 \\
-2 & 3
\end{pmatrix} 
\bmod{9}
\Big\}.
\end{align*}
Moreover, 
as a coset decomposition of $\Gamma_{\theta,3}$ modulo $\Ker \nu_{F^{k}}$,  
we may choose $S^{n}$ $(n=0,3,6),$ where 
$
S=
\begin{pmatrix}
1 & 1 \\
0 & 1
\end{pmatrix}. 
$
}
\end{theorem}

\begin{proof}
For each $M\in\Gamma_{\theta,3},$ 
we have 
\begin{equation*}
\nu_{F^k}(M)
=
\exp
\left\{
\pm
\frac{2\pi i}{3}
f(M)
\right\}. 
\end{equation*}
The theorem follows from Lemma \ref{lem:level3-mod3}.  
\par
The coset decomposition can be obtained by considering the values of $\nu_{F^{k}  }(M).$
\end{proof}

\subsection{Case where $k\equiv \pm 2 \bmod{12}$}

\begin{theorem}
\label{thmk=2-level3}
{\it
Suppose that 
$k\equiv \pm2 \bmod{12}$. 
Then, it follows that 
\begin{align*}
\Ker \nu_{F^{k}}
=
\Ker 
\nu_{ F^{6} }
\cap
\Ker \nu_{ F^4 }. 
\end{align*}
Moreover, 
as a coset decomposition of $\Gamma_{\theta,3}$ modulo $\Ker \nu_{F^{k}}$,  
we may choose 
\begin{equation*}
S^{n}  \,\, 
(n=0,3,6,9,12,15),  \,\,
S=
\begin{pmatrix}
1 & 1 \\
0 & 1
\end{pmatrix}. 
\end{equation*}
}
\end{theorem}

\begin{proof}
For each 
$
M
=
\begin{pmatrix}
a & b \\
c & d
\end{pmatrix}
\in
\Gamma_{\theta,3}, 
$ 
we have 
\begin{equation*}
\nu_{F^{k}  }(M)=
\exp
\left\{
\pm
\frac{
 \pi i
}
{3} 
f(M) 
\right\},
\end{equation*}
which implies that 
\begin{equation*}
M
\in
\Ker \nu_{F^{k}}
\Longleftrightarrow 
f(M)\equiv 0 \bmod{6}
\Longleftrightarrow 
f(M)\equiv 0 \bmod{2} \,\,
\mathrm{and} \,\,
f(M)\equiv 0 \bmod{3}. 
\end{equation*}
The theorem follows from Theorems \ref{thmk=6-level3} and \ref{thmk=4-level3}.
\par
The coset decomposition can be obtained by considering the values of $\nu_{F^{k}  }(M).$
\end{proof}

\subsection{Case where $k\equiv \pm 1, \pm 5 \bmod{12}$}

\begin{theorem}
\label{thmk=1and5}
{\it
Suppose that 
$k\equiv \pm1, \pm 5 \bmod{12}$. 
Then, it follows that 
\begin{align*}
\Ker \nu_{F^{k}}
=
\Ker 
\nu_{ F^{3} }
\cap
\Ker \nu_{ F^4 }. 
\end{align*}
Moreover, 
as a coset decomposition of $\Gamma_{\theta,3}$ modulo $\Ker \nu_{F^{k}}$,  
we may choose 
\begin{equation*}
S^{n}  \,\, 
(n=0,3,6,9,12,15,18,21,24,27,30,33),  \,\,
S=
\begin{pmatrix}
1 & 1 \\
0 & 1
\end{pmatrix}. 
\end{equation*}
}
\end{theorem}

\begin{proof}
For each 
$
M
=
\begin{pmatrix}
a & b \\
c & d
\end{pmatrix}
\in
\Gamma_{\theta,3}, 
$ 
we have 
\begin{equation*}
\nu_{ F^{k}  }(M)=
\exp
\left\{
\pm
\frac{
k
 \pi i
}
{6} 
f(M) 
\right\}, \,\,
(k,6)=1, 
\end{equation*}
which implies that 
\begin{equation*}
M
\in
\Ker \nu_{F^{k}}
\Longleftrightarrow 
f(M)\equiv 0 \bmod{12}
\Longleftrightarrow 
f(M)\equiv 0 \bmod{4} \,\,
\mathrm{and} \,\,
f(M)\equiv 0 \bmod{3}. 
\end{equation*}
The theorem follows from Theorems \ref{thmk=3-level3} and \ref{thmk=4-level3}.
\par
The coset decomposition can be obtained by considering the values of $\nu_{ F^{k}  }(M).$
\end{proof}

\section{Modular form on $\Gamma_{\theta,4}$ }
\label{sec:level4}

We note 
\begin{equation*}
\Gamma_{\theta,4}
=
\left\{
\begin{pmatrix}
a & b \\
c & d
\end{pmatrix}
\in
\Gamma(1) \,\Big| \,
\begin{pmatrix}
a & b \\
c & d
\end{pmatrix}
\equiv 
\pm
\begin{pmatrix}
1 & 0 \\
0 & 1
\end{pmatrix}, 
\pm
\begin{pmatrix}
1 & 2 \\
2 & 1
\end{pmatrix}, 
\pm
\begin{pmatrix}
0 & -1 \\
1 & 0
\end{pmatrix} 
\pm
\begin{pmatrix}
2 & -1 \\
1 & 2
\end{pmatrix} 
\bmod{4}
\right\}.
\end{equation*}

\begin{proposition}
\label{prop:character-level4}
{\it
For every $\tau\in\mathscr{H},$
set 
\begin{equation*}
G(\tau)= \eta \left(\frac{\tau-1}{4} \right) \eta\left(\frac{\tau+1}{4} \right).
\end{equation*}
The multiplier system $\nu_{G}$ is a character. 
}
\end{proposition}

\begin{proof}
Set 
$
M=
\begin{pmatrix}
a & b \\
c & d
\end{pmatrix}
\in 
\Gamma_{\theta,4}. 
$ 
We note that $\nu_{G}$ is a character if the weight $r$ is a rational integer.  
\par
We first suppose that 
$
M
\equiv \pm
\begin{pmatrix}
1 & 0 \\
0 & 1
\end{pmatrix}, 
\pm
\begin{pmatrix}
1 & 2 \\
2 & 1
\end{pmatrix}
\bmod{4},$ 
which implies that  
\begin{align*}
G(M\tau)=&
\eta \left(\frac{ M \tau-1}{4} \right) \eta\left(\frac{ M\tau+1}{4} \right)
=
\eta \left( M_1 \left( \frac{  \tau-1}{4} \right)  \right) \eta\left( M_2  \left( \frac{ \tau+1}{4} \right)  \right) \\
=&
\nu_{\eta}(M_1) 
\left(
4c \cdot 
\frac{\tau-1}{4}
+(d+c)
\right)^{\frac12}
\eta \left(\frac{\tau-1}{4} \right)
\nu_{\eta} (M_2)
\left(
4c \cdot 
\frac{\tau+1}{4}
+(d-c)
\right)^{\frac12}
\eta\left(\frac{\tau+1}{4} \right)  \\
=&
\nu_{\eta}(M_1) \nu_{\eta} (M_2)
(c\tau+d)G(\tau),
\end{align*}
where 
\begin{equation*}
M_1
=
\begin{pmatrix}
a-c & \frac{ b-d+a-c }{4} \\
4c & d+c
\end{pmatrix}, 
M_2
=
\begin{pmatrix}
a+c & \frac{ b+d-a-c }{4} \\
4c & d-c
\end{pmatrix}. 
\end{equation*}
\par
We next assume that 
$
M
\equiv \pm
\begin{pmatrix}
0 & -1 \\
1 & 0
\end{pmatrix},
\pm
\begin{pmatrix}
2 & -1 \\
1 & 2
\end{pmatrix} 
\bmod{4},
$ 
which implies that  
\begin{align*}
G(M\tau)=&
\eta \left(\frac{ M \tau-1}{4} \right) \eta\left(\frac{ M\tau+1}{4} \right)
=
\eta \left( M_1 \left( \frac{  \tau+1}{4} \right)  \right) \eta\left( M_2  \left( \frac{ \tau-1}{4} \right)  \right) \\
=&
\nu_{\eta}(M_1) 
\left(
4c \cdot 
\frac{\tau+1}{4}
+(d-c)
\right)^{\frac12}
\eta \left(\frac{\tau+1}{4} \right) 
\nu_{\eta} (M_2)
\left(
4c \cdot 
\frac{\tau-1}{4}
+(d+c)
\right)^{\frac12}
\eta\left(\frac{\tau-1}{4} \right)  \\
=&
\nu_{\eta}(M_1) \nu_{\eta} (M_2)
(c\tau+d)G(\tau),
\end{align*}
where 
\begin{equation*}
M_1
=
\begin{pmatrix}
a-c & \frac{ b-d-a+c }{4} \\
4c & d-c
\end{pmatrix}, 
M_2
=
\begin{pmatrix}
a+c & \frac{ b+d+a+c }{4} \\
4c & d+c
\end{pmatrix}. 
\end{equation*}
\end{proof}

\subsection{Case where 
$
M
\equiv \pm
\begin{pmatrix}
0 & -1 \\
1 & 0
\end{pmatrix}, 
\pm
\begin{pmatrix}
2 & -1 \\
1 & 2
\end{pmatrix}
\bmod{4}
$
}

\begin{lemma}
\label{lem:(c/d)^*}
{\it
Suppose that $c$ is odd and $d$ is even and 
$(c,d)=1.$ 
Then, 
we have 
\begin{equation*}
\left(
\frac{c}{d-c}
\right)_{*}
\left(
\frac{c}{d+c}
\right)_{*}
=
(-1)^{
\frac{c-1}{2}
}. 
\end{equation*}
}
\end{lemma}

\begin{proof}
By the definition, 
we have 
\begin{equation*}
\left(
\frac{c}{d-c}
\right)_{*}
\left(
\frac{c}{d+c}
\right)_{*}
=
\left(
\frac{c}{ |d-c| }
\right)
\left(
\frac{c}{ |d+c| }
\right)
(-1)^{
\frac{\sign c -1}{2}
\cdot
\frac{\sign (d-c) -\sign(d+c)}{2}
}. 
\end{equation*}
We note that 
\begin{equation*}
(-1)^{
\frac{\sign c -1}{2}
\cdot
\frac{\sign (d-c) -\sign(d+c)}{2}
} 
=
-1
\Longleftrightarrow 
c<0 \,\, 
 \mathrm{and} 
 \,\, 
 |d|<|c|,
\end{equation*}
which implies that 
\begin{align*}
\left(
\frac{c}{d-c}
\right)_{*}
\left(
\frac{c}{d+c}
\right)_{*}
=&
\sign c
\left(
\frac{ |c| }{|d^2-c^2|}
\right)
=
\sign c
\left(
\frac {|d^2-c^2|}{ |c| }
\right)
(-1)^
{
\frac{ |c|-1}{2}
\cdot
\frac{ |d^2-c^2|-1  }{2}
}  \\
=&
\sign c 
\cdot
(-1)^{
\frac{|c|-1}{2}
}
=
(-1)^{
\frac{c-1}{2}
}. 
\end{align*}
\end{proof}

\begin{theorem}
\label{thm:M:c-odd-level4}
{\it
Suppose that 
$
M=
\begin{pmatrix}
a & b \\
c & d
\end{pmatrix}
\in 
\Gamma_{\theta,4}
$ 
and 
$M\equiv 
\pm
\begin{pmatrix}
0 & -1 \\
1 & 0
\end{pmatrix}, 
 \pm
\begin{pmatrix}
2 & -1 \\
1 & 2
\end{pmatrix}
\bmod{4}. 
$ 
Then, 
the multiplier system 
$\nu_{G}$ of 
$
\displaystyle 
G=
\eta \left(\frac{\tau-1}{4} \right) \eta\left(\frac{\tau+1}{4} \right)
$ 
is given by 
\begin{equation*}
\nu_{G}
(M)
=
\exp
\left\{
\frac{\pi i}{6}
\left[
3(c+d-2)+\frac14(a+d)c
+
\frac14(b+c)d-cd(1+bc)
\right]
\right\}. 
\end{equation*}
In particular, we have $\nu_{G}(T)=-i,$ where 
$
T=
\begin{pmatrix}
0 & -1 \\
1 & 0
\end{pmatrix}.  
$
}
\end{theorem}

\begin{proof}
We note that 
\begin{align*}
G(M\tau)=&
\nu_{\eta}(M_1) \nu_{\eta} (M_2)
(c\tau+d)G(\tau),
\end{align*}
where 
\begin{equation*}
M_1
=
\begin{pmatrix}
a-c & \frac{ b-d-a+c }{4} \\
4c & d-c
\end{pmatrix}, 
M_2
=
\begin{pmatrix}
a+c & \frac{ b+d+a+c }{4} \\
4c & d+c
\end{pmatrix}. 
\end{equation*}
Equation (\ref{eqn-eta-multiplier})
yields 
\begin{align*}
\nu_{\eta}(M_1)=&
\left(
\frac{4c}{d-c}
\right)_{*} 
\exp
\Big\{
\frac{\pi i}{12}
\Big[
-4 a c^3+4 a c^2 d+4 b c^3-4 b c^2 d+4 c^4-8 c^3 d+4 c^2 d^2-3 c+3 d-3  \\
&\hspace{40mm}
+\frac{17 a c}{4}-\frac{a d}{4}-\frac{b c}{4}+\frac{b d}{4}+\frac{15 c^2}{4}-\frac{15 c d}{2}-\frac{d^2}{4}
\Big]
\Big\},
\end{align*}
and 
\begin{align*}
\nu_{\eta}(M_2)=&
\left(
\frac{4c}{d+c}
\right)_{*}  
\exp 
\Big\{
\frac{\pi i}{12}
\Big[
-4 a c^3-4 a c^2 d-4 b c^3-4 b c^2 d-4 c^4-8 c^3 d-4 c^2 d^2+3 c+3 d-3 \\
&\hspace{40mm}
+\frac{17 a c}{4}+\frac{a d}{4}+\frac{b c}{4}+\frac{b d}{4}-\frac{15 c^2}{4}-\frac{15 c d}{2}+\frac{d^2}{4}
\Big]
\Big\},
\end{align*}
which implies that 
\begin{equation*}
\nu_{G}
=
\nu_{\eta}(M_1) \nu_{\eta} (M_2)
=
\left(
\frac{c}{d-c}
\right)_{*}  
\left(
\frac{c}{d+c}
\right)_{*}  
\exp
\left\{
\frac{\pi i}{12}
E
\right\}, 
\end{equation*}
where 
\begin{equation*}
E=
-6+\frac{17 a c}{2}-8 a c^3+6 d+\frac{b d}{2}-15 c d-8 b c^2 d-16 c^3 d. 
\end{equation*}
\par
Because $c(c^2-1)\equiv0 \bmod{3}$ and $c$ is odd, 
it follows that 
$c^2\equiv 1 \bmod{8},$ 
which implies 
\begin{align*}
E=&
-6+8ac(1-c^2)+\frac12ac+6d+\frac12bd-16cd(c^2-1)-31cd-8bc^2d & \\
\equiv & \,
-6+\frac12ac+6d+\frac12bd-31cd-8bc^2d
& &\bmod{24} \,  \\
\equiv &\, 
-6+\frac12ac+6d+\frac12bd-7cd-6bc^2d-2bc^2d &  &\bmod{24} \, \\
\equiv &
\, 
-6+\frac12ac+6d+\frac12bd-cd-6(b+c)d-2bc^2d
&  &\bmod{24}  \\
\equiv &\,
-6+\frac12ac+6d+\frac12bd-cd-2bc^2d  &  &\bmod{24} \,  \\
\equiv & \,
2
\left[
3(d-1)+\frac14(a+d)c+\frac14(b+c)d-cd(1+bc)
\right]    &  &\bmod{24},
\end{align*}
which proves the theorem. 
\end{proof}

\subsection{Case where 
$
M
\equiv \pm
\begin{pmatrix}
1 & 0 \\
0 & 1
\end{pmatrix}, 
\pm
 \begin{pmatrix}
1 & 2 \\
2 & 1
\end{pmatrix}
\bmod{4}
$ 
}

\begin{theorem}
\label{thm:M:c-even-level4}
{\it
Suppose that 
$
M=
\begin{pmatrix}
a & b \\
c & d
\end{pmatrix}
\in 
\Gamma_{\theta,4}
$ 
and 
$M\equiv 
 \pm
\begin{pmatrix}
1 & 0 \\
0 & 1
\end{pmatrix}, 
\pm
 \begin{pmatrix}
1 & 2 \\
2 & 1
\end{pmatrix}
\bmod{4}.
$ 
Then, 
the multiplier system 
$\nu_{G}$ of 
$
\displaystyle 
G=
\eta \left(\frac{\tau-1}{4} \right) \eta\left(\frac{\tau+1}{4} \right)
$ 
is given by 
\begin{equation*}
\nu_{G}
(M)
=
\exp
\left\{
\frac{\pi i}{6}
\left[
3(d-c-1)+\frac14(b-c)d
+\frac14(a-d)c+cd(1-ad)
\right]
\right\}. 
\end{equation*}
}
\end{theorem}

\begin{proof}
By 
$
T
=
\begin{pmatrix}
0 & -1 \\
1 & 0
\end{pmatrix}, 
$ 
we have 
\begin{equation*}
MT
=
\begin{pmatrix}
a & b \\
c & d
\end{pmatrix}
\begin{pmatrix}
0 & -1 \\
1 & 0
\end{pmatrix}
=
\begin{pmatrix}
b & -a \\
d & -c
\end{pmatrix}
\equiv
\pm
\begin{pmatrix}
0 & -1 \\
1 & 0
\end{pmatrix},
\pm
 \begin{pmatrix}
1 & 2 \\
2 & 1
\end{pmatrix}
\bmod{4}.
\end{equation*}
\par
Proposition \ref{prop:character-level4} and Theorem \ref{thm:M:c-odd-level4}
imply that 
\begin{align*}
\nu_{G}(MT)=&\nu_{F}(M)\nu_{F}(T)=\nu_{F}(M)\cdot(-i)  \\
=&
\exp
\left\{
\frac{\pi i}{6}
\left[
3(d-c-2)+\frac14(b-c)d+\frac14(a-d)c+cd(1-ad)
\right]
\right\},
\end{align*}
which proves the theorem. 
\end{proof}

\subsection{Summary}

\begin{theorem}
\label{summary-level4}
{\it
Set
$
M=
\begin{pmatrix}
a & b \\
c & d
\end{pmatrix}
\in 
\Gamma_{\theta,4}. 
$ 
Then, it follows that 
the multiplier system 
$\nu_{G}$ of 
 $
\displaystyle 
G=
\eta \left(\frac{\tau-1}{4} \right) \eta\left(\frac{\tau+1}{4} \right) 
$ 
is given by 
\begin{equation*}
\nu_{G}(M)=
\begin{cases}
\exp
\left\{
\displaystyle
\frac{\pi i}{6}
\left[
3(c+d-2)+\frac14(a+d)c+\frac14(b+c)d-cd(1+bc)
\right]
\right\}   
&
\text
{ 
if
$c$ is odd,
} \vspace{1mm}  \\
\exp
\left\{
\displaystyle
\frac{\pi i}{6}
\left[
3(d-c-1)+\frac14(b-c)d+\frac14(a-d)c+cd(1-ad)
\right]
\right\}   
&
\text
{ 
if
$c$ is even. 
}  
\end{cases}
\end{equation*}
}
\end{theorem}

\begin{proof}
The theorem follows from 
Theorems 
\ref{thm:M:c-odd-level4} 
and 
\ref{thm:M:c-even-level4}.
\end{proof}

\begin{definition}
For each 
$
M=
\begin{pmatrix}
a & b \\
c & d
\end{pmatrix}
\in 
\Gamma_{\theta,4},
$ 
we define 
\begin{equation*}
g(M)
=
\begin{cases}
\displaystyle
3(c+d-2)+\frac14(a+d)c+\frac14(b+c)d-cd(1+bc)
&
\text
{ 
if
$c$ is odd,
}  \vspace{1mm} \\
\displaystyle
3(d-c-1)+\frac14(b-c)d+\frac14(a-d)c+cd(1-ad)  
&
\text
{ 
if
$c$ is even. 
}  
\end{cases}
\end{equation*}
\end{definition}

From the definition, 
it follows that, for each $k\in\mathbb{Z},$ 
\begin{equation*}
\nu_{G^k}(M)=
\exp \left\{
\frac{\pi i k}{6}
g(M)
\right\}. 
\end{equation*}

\subsection{Case where $k\equiv 0 \bmod{12}$}

\begin{theorem}
{\it
Suppose that 
$k\equiv 0 \bmod{12}$. 
Then, it follows that 
\begin{equation*}
\Ker \nu_{G^{k}}
=\Gamma_{\theta,4}.
\end{equation*}
}
\end{theorem}

\begin{proof}
It is obvious. 
\end{proof}

\subsection{Case where $k\equiv 6 \bmod{12}$}

\begin{lemma}
\label{lem:level4-mod2}
{\it
Set 
$
M=
\begin{pmatrix}
a & b \\
c & d
\end{pmatrix}
\in\Gamma_{\theta,4}.
$
\newline
$\mathrm{(1)}$ If 
$
M
\equiv
\pm
\begin{pmatrix}
1 & 0 \\
0 & 1
\end{pmatrix}
\bmod{4},
$ 
we have 
\begin{align}
&g(M)\equiv 0 \bmod{2}  
\Longleftrightarrow
\, b-c \equiv 0 \bmod{8} \notag  \\
\Longleftrightarrow&
M
\equiv
\pm
\begin{pmatrix}
1 & 0 \\
0 & 1
\end{pmatrix}, 
\pm
\begin{pmatrix}
1 & 4 \\
4 & 1
\end{pmatrix}, 
\pm
\begin{pmatrix}
-3 & 0 \\
0 & -3
\end{pmatrix}, 
\pm
\begin{pmatrix}
-3 & 4 \\
4 & -3
\end{pmatrix}
\mod{8}.       \label{eqn:level4-mod2-pmI(1)}
\end{align}
$\mathrm{(2)}$ If 
$
M
\equiv
\pm
\begin{pmatrix}
1 & 2 \\
2 & 1
\end{pmatrix}
\bmod{4}
$, we have  
\begin{align}
 \label{eqn:level4-mod4-pmI(2)}
&g(M)\equiv 0 \bmod{2}  
\Longleftrightarrow
\, b-c\equiv 0 \bmod{8} \notag \\
\Longleftrightarrow& \, 
M
\equiv 
\pm
\begin{pmatrix}
1 & 2 \\
2 & -3
\end{pmatrix}, 
\pm
\begin{pmatrix}
1 & -2 \\
-2 & -3
\end{pmatrix}, 
\pm
\begin{pmatrix}
-3 & 2 \\
2 & 1
\end{pmatrix}, 
\pm
\begin{pmatrix}
-3 & -2 \\
-2 & 1
\end{pmatrix} 
\bmod{8}.
\end{align}
$\mathrm{(3)}$ If 
$
M
\equiv
\pm
\begin{pmatrix}
0 & -1 \\
1 & 0
\end{pmatrix}
\bmod{4}
$,  
we have 
\begin{align}
&g(M)\equiv 0 \bmod{2} 
\Longleftrightarrow
\, a+d\equiv 4 \bmod{8}  \notag \\
\Longleftrightarrow&
M
\equiv
\pm
\begin{pmatrix}
0 & -1 \\
1 & 4
\end{pmatrix}, 
\pm
\begin{pmatrix}
0 & 3 \\
-3 & 4
\end{pmatrix}, 
\pm
\begin{pmatrix}
4 & -1 \\
1 & 0
\end{pmatrix}, 
\pm
\begin{pmatrix}
4 & 3 \\
-3 & 0
\end{pmatrix}
 \bmod{8}.   \label{eqn:level4-mod4-pmT}
\end{align}
$\mathrm{(4)}$ If 
$
M
\equiv
\pm
\begin{pmatrix}
2 & -1 \\
1 & 2
\end{pmatrix}
\bmod{4}
$, we have  
\begin{align}
\label{eqn:level3-mod4-pmT-c-even}
&g(M)\equiv 0 \bmod{2} 
\Longleftrightarrow \,
a+d\equiv 4 \bmod{8} \notag \\
\Longleftrightarrow&
M
\equiv 
\pm
\begin{pmatrix}
2 & -1 \\
-3 & 2
\end{pmatrix}, 
\pm
\begin{pmatrix}
2 & 3 \\
1 & 2
\end{pmatrix}, 
\pm
\begin{pmatrix}
-2 & -1 \\
-3 & -2
\end{pmatrix}, 
\pm
\begin{pmatrix}
-2 & 3 \\
1 & -2
\end{pmatrix} 
\bmod{8}.
\end{align}
}
\end{lemma}

\begin{proof}
We treat case (1). The other cases can be proved in the same way. 
From the definition of $g(M),$ it follows that 
\begin{align*}
g(M)\equiv 0 \bmod{2} &\Longleftrightarrow 
3(d-c-1)+\frac14(b-c)d+\frac14(a-d)c+cd(1-ad)  \equiv 0 \bmod{2}   \\
&\Longleftrightarrow
\frac14(b-c)\equiv 0 \bmod{2}  
\Longleftrightarrow
b-c\equiv 0 \bmod{8}   \\
&\Longleftrightarrow
(b,c)\equiv (0,0), (4,4) \bmod{8}. 
\end{align*}
Because $ad-bc=1,$ it follows that 
$ad\equiv 1 \bmod{8},$ 
which implies that 
\begin{equation*}
(a,d)\equiv 
(1,1), (-3,-3), (-1,-1), (3,3) \bmod{8}.
\end{equation*}
\end{proof}

\begin{theorem}
\label{thmk=6-level4}
{\it
Suppose that 
$k\equiv 6 \bmod{12}$. 
Then, it follows that 
\begin{align*}
\Ker \nu_{G^{k}}
=&
\left\{
\begin{pmatrix}
a & b \\
c & d
\end{pmatrix}
\in
\Gamma_{\theta,4} \,\Big| \, 
b-c\equiv 0 \bmod{8} \, (b, c\in2 \mathbb{Z}) \,\mathrm{or} \,
a+d\equiv 4 \bmod{8}  \, (a,d\in2\mathbb{Z}) 
\right\}.  \\
=&
\Bigg\{
\begin{pmatrix}
a & b \\
c & d
\end{pmatrix}
\in
\Gamma_{\theta,4} \,\Big| \, 
\begin{pmatrix}
a & b \\
c & d
\end{pmatrix}
\equiv
\pm
\begin{pmatrix}
1 & 0 \\
0 & 1
\end{pmatrix}, 
\pm
\begin{pmatrix}
1 & 4 \\
4 & 1
\end{pmatrix}, 
\pm
\begin{pmatrix}
-3 & 0 \\
0 & -3
\end{pmatrix}, 
\pm
\begin{pmatrix}
-3 & 4 \\
4 & -3
\end{pmatrix},  \\
&\hspace{52mm}
\pm
\begin{pmatrix}
1 & 2 \\
2 & -3
\end{pmatrix}, 
\pm
\begin{pmatrix}
1 & -2 \\
-2 & -3
\end{pmatrix}, 
\pm
\begin{pmatrix}
-3 & 2 \\
2 & 1
\end{pmatrix}, 
\pm
\begin{pmatrix}
-3 & -2 \\
-2 & 1
\end{pmatrix}, \\
&\hspace{52mm}
\pm
\begin{pmatrix}
0 & -1 \\
1 & 4
\end{pmatrix}, 
\pm
\begin{pmatrix}
0 & 3 \\
-3 & 4
\end{pmatrix}, 
\pm
\begin{pmatrix}
4 & -1 \\
1 & 0
\end{pmatrix}, 
\pm
\begin{pmatrix}
4 & 3 \\
-3 & 0
\end{pmatrix},  \\
&\hspace{52mm}
\pm
\begin{pmatrix}
2 & -1 \\
-3 & 2
\end{pmatrix}, 
\pm
\begin{pmatrix}
2 & 3 \\
1 & 2
\end{pmatrix}, 
\pm
\begin{pmatrix}
-2 & -1 \\
-3 & -2
\end{pmatrix}, 
\pm
\begin{pmatrix}
-2 & 3 \\
1 & -2
\end{pmatrix} 
\bmod{8}
\Bigg\}.
\end{align*}
Moreover, 
as a coset decomposition of $\Gamma_{\theta,4}$ modulo $\Ker \nu_{G^{k}}$,  
we may choose $S^{n}$ $(n=0,4),$ where 
$
S=
\begin{pmatrix}
1 & 1 \\
0 & 1
\end{pmatrix}. 
$
}
\end{theorem}

\begin{proof}
For each $M\in\Gamma_{\theta,4},$ 
we have 
\begin{equation*}
\nu_{G^k}(M)=
\exp
\left\{
\pi i
\left[
g(M)
\right]
\right\}. 
\end{equation*}
The theorem follows from Lemma \ref{lem:level4-mod2}.  
\par
The coset decomposition can be obtained by considering the values of $\nu_{G^{k}  }(M).$
\end{proof}

\subsection{Case where $k\equiv \pm 3 \bmod{12}$}

\begin{lemma}
\label{lem:level4-mod4}
{\it
Set 
$
M=
\begin{pmatrix}
a & b \\
c & d
\end{pmatrix}
\in\Gamma_{\theta,4}.
$
\newline
$\mathrm{(1)}$ If 
$
M
\equiv
\pm
\begin{pmatrix}
1 & 0 \\
0 & 1
\end{pmatrix}
\bmod{4}
$, 
we have 
\begin{align}
&\hspace{10mm} g(M)\equiv 0 \bmod{4}  
\Longleftrightarrow
\, \frac{b-c}{4}\equiv -d+1 \bmod{4} \notag  \\
&\Longleftrightarrow
M
\equiv
\begin{pmatrix}
1 & 0 \\
0 & 1
\end{pmatrix}, 
\begin{pmatrix}
1 & 4 \\
4 & 1
\end{pmatrix}, 
\begin{pmatrix}
1 & -4 \\
-4 & 1
\end{pmatrix}, 
\begin{pmatrix}
1 & 8 \\
8 & 1
\end{pmatrix},  
\begin{pmatrix}
5 & 0 \\
0 & -3
\end{pmatrix}, 
\begin{pmatrix}
5 & 4 \\
4 & -3
\end{pmatrix}, 
\begin{pmatrix}
5 & -4 \\
-4 & -3
\end{pmatrix}, 
\begin{pmatrix}
5 & 8 \\
8 & -3
\end{pmatrix}, \notag  \\
&\hspace{10mm}
\begin{pmatrix}
-7 & 0 \\
0 & -7
\end{pmatrix}, 
\begin{pmatrix}
-7 & 4 \\
4 & -7
\end{pmatrix}, 
\begin{pmatrix}
-7 & -4 \\
-4 & -7
\end{pmatrix}, 
\begin{pmatrix}
-7 & 8 \\
8 & -7
\end{pmatrix},  
\begin{pmatrix}
-3 & 0 \\
0 & 5
\end{pmatrix}, 
\begin{pmatrix}
-3 & 4 \\
4 & 5
\end{pmatrix}, 
\begin{pmatrix}
-3 & -4 \\
-4 & 5
\end{pmatrix}, 
\begin{pmatrix}
-3 & 8 \\
8 & 5
\end{pmatrix},    \notag  \\
&\hspace{10mm}
\begin{pmatrix}
-1 & 0 \\
8 & -1
\end{pmatrix}, 
\begin{pmatrix}
-1 & 4 \\
-4 & -1
\end{pmatrix}, 
\begin{pmatrix}
-1 & 8 \\
0 & -1
\end{pmatrix}, 
\begin{pmatrix}
-1 & -4 \\
4 & -1
\end{pmatrix},   
\begin{pmatrix}
3 & 0 \\
8 & -5
\end{pmatrix}, 
\begin{pmatrix}
3 & 4 \\
-4 & -5
\end{pmatrix}, 
\begin{pmatrix}
3 & 8 \\
0 & -5
\end{pmatrix}, 
\begin{pmatrix}
3 & -4 \\
4 & -5
\end{pmatrix},   \notag \\
&\hspace{10mm}
\begin{pmatrix}
7 & 0 \\
8 & 7
\end{pmatrix}, 
\begin{pmatrix}
7 & 4 \\
-4 & 7
\end{pmatrix}, 
\begin{pmatrix}
7 & 8 \\
0 & 7
\end{pmatrix}, 
\begin{pmatrix}
7 & -4 \\
4 & 7
\end{pmatrix},   
\begin{pmatrix}
-5 & 0 \\
8 & 3
\end{pmatrix}, 
\begin{pmatrix}
-5 & 4 \\
-4 & 3
\end{pmatrix}, 
\begin{pmatrix}
-5 & 8 \\
0 & 3
\end{pmatrix}, 
\begin{pmatrix}
-5 & -4 \\
4 & 3
\end{pmatrix}  
\bmod{16}.  \notag
\end{align}
$\mathrm{(2)}$ If 
$
M
\equiv
\pm
\begin{pmatrix}
1 & 2 \\
2 & 1
\end{pmatrix}
\bmod{4},
$ 
we have  
\begin{align*}
&\hspace{10mm} g(M)\equiv 0 \bmod{4} 
\Longleftrightarrow
\, \frac{b-c}{4}\equiv d-1 \bmod{4} \notag \\
&\Longleftrightarrow \, 
M
\equiv 
\begin{pmatrix}
1 & 2 \\
2 & 5
\end{pmatrix}, 
\begin{pmatrix}
5 & 2 \\
2 & 1
\end{pmatrix}, 
\begin{pmatrix}
-7 & 2 \\
2 & -3
\end{pmatrix}, 
\begin{pmatrix}
-3 & 2 \\
2 & -7
\end{pmatrix}, 
\begin{pmatrix}
1 & 6 \\
6 & 5
\end{pmatrix}, 
\begin{pmatrix}
5 & 6 \\
6 & 1
\end{pmatrix}, 
\begin{pmatrix}
-7 & 6 \\
6 & -3
\end{pmatrix}, 
\begin{pmatrix}
-3 & 6 \\
6 & -7
\end{pmatrix},   \\
&\hspace{8mm}
\begin{pmatrix}
1 & -6 \\
-6 & 5
\end{pmatrix}, 
\begin{pmatrix}
5 & -6 \\
-6 & 1
\end{pmatrix}, 
\begin{pmatrix}
-7 & -6 \\
-6 & -3
\end{pmatrix}, 
\begin{pmatrix}
-3 & -6 \\
-6 & -7
\end{pmatrix},  
\begin{pmatrix}
1 & -2 \\
-2 & 5
\end{pmatrix}, 
\begin{pmatrix}
5 & -2 \\
-2 & 1
\end{pmatrix}, 
\begin{pmatrix}
-7 & -2 \\
-2 & -3
\end{pmatrix}, 
\begin{pmatrix}
-3 & -2 \\
-2 & -7
\end{pmatrix},  \\
&\hspace{8mm}
\begin{pmatrix}
3 & 2 \\
-6 & 7
\end{pmatrix}, 
\begin{pmatrix}
7 & 2 \\
-6 & 3
\end{pmatrix}, 
\begin{pmatrix}
-5 & 2 \\
-6 & -1
\end{pmatrix}, 
\begin{pmatrix}
-1 & 2 \\
-6 & -5
\end{pmatrix}, 
\begin{pmatrix}
3 & 6 \\
-2 & 7
\end{pmatrix}, 
\begin{pmatrix}
7 & 6 \\
-2 & 3
\end{pmatrix}, 
\begin{pmatrix}
-5 & 6 \\
-2 & -1
\end{pmatrix}, 
\begin{pmatrix}
-1 & 6 \\
-2 & -5
\end{pmatrix},   \\
&\hspace{8mm}
\begin{pmatrix}
3 & -6 \\
2 & 7
\end{pmatrix}, 
\begin{pmatrix}
7 & -6 \\
2 & 3
\end{pmatrix}, 
\begin{pmatrix}
-5 & -6 \\
2 & -1
\end{pmatrix}, 
\begin{pmatrix}
-1 & -6 \\
2 & -5
\end{pmatrix}, 
\begin{pmatrix}
3 & -2 \\
6 & 7
\end{pmatrix}, 
\begin{pmatrix}
7 & -2 \\
6 & 3
\end{pmatrix}, 
\begin{pmatrix}
-5 & -2 \\
6 & -1
\end{pmatrix}, 
\begin{pmatrix}
-1 & -2 \\
6 & -5
\end{pmatrix}
\bmod{16}.
\end{align*}
$\mathrm{(3)}$ If 
$
M
\equiv
\pm
\begin{pmatrix}
0 & -1 \\
1 & 0
\end{pmatrix}
\bmod{4}
$, 
we have 
\begin{align}
&\hspace{10mm}  g(M)\equiv 0 \bmod{4} 
\Longleftrightarrow
\, \frac{a+d}{4}\equiv -1 \bmod{4}  \notag \\
&\Longleftrightarrow
M
\equiv
\begin{pmatrix}
0 & 3 \\
5 & -4
\end{pmatrix}, 
\begin{pmatrix}
0 & 7 \\
-7 & -4
\end{pmatrix}, 
\begin{pmatrix}
0 & -5 \\
-3 & -4
\end{pmatrix}, 
\begin{pmatrix}
0 & -1 \\
1 & -4
\end{pmatrix},  
\begin{pmatrix}
4 & 3 \\
5 & 8
\end{pmatrix}, 
\begin{pmatrix}
4 & 7 \\
-7 & 8
\end{pmatrix}, 
\begin{pmatrix}
4 & -5 \\
-3 & 8
\end{pmatrix}, 
\begin{pmatrix}
4 & -1 \\
1 & 8
\end{pmatrix}, \notag \\
&\hspace{10mm}
\begin{pmatrix}
8 & 3 \\
5 & 4
\end{pmatrix}, 
\begin{pmatrix}
8 & 7 \\
-7 & 4
\end{pmatrix}, 
\begin{pmatrix}
8 & -5 \\
-3 & 4
\end{pmatrix}, 
\begin{pmatrix}
8 & -1 \\
1 & 4
\end{pmatrix}, 
\begin{pmatrix}
-4 & 3 \\
5 & 0
\end{pmatrix}, 
\begin{pmatrix}
-4 & 7 \\
-7 & 0
\end{pmatrix}, 
\begin{pmatrix}
-4 & -5 \\
-3 & 0
\end{pmatrix}, 
\begin{pmatrix}
-4 & -1 \\
1 & 0
\end{pmatrix},  \notag  \\
&\hspace{10mm}
\begin{pmatrix}
0 & 1 \\
-1 & -4
\end{pmatrix}, 
\begin{pmatrix}
0 & 5 \\
3 & -4
\end{pmatrix}, 
\begin{pmatrix}
0 & -7 \\
7 & -4
\end{pmatrix}, 
\begin{pmatrix}
0 & -3 \\
-5 & -4
\end{pmatrix},  
\begin{pmatrix}
4 & 1 \\
-1 & 8
\end{pmatrix}, 
\begin{pmatrix}
4 & 5 \\
3 & 8
\end{pmatrix}, 
\begin{pmatrix}
4 & -7 \\
7 & 8
\end{pmatrix}, 
\begin{pmatrix}
4 & -3 \\
-5 & 8
\end{pmatrix},  \notag \\  
&\hspace{10mm}
\begin{pmatrix}
8 & 1 \\
-1 & 4
\end{pmatrix}, 
\begin{pmatrix}
8 & 5 \\
3 & 4
\end{pmatrix}, 
\begin{pmatrix}
8 & -7 \\
7 & 4
\end{pmatrix}, 
\begin{pmatrix}
8 & -3 \\
-5 & 4
\end{pmatrix}, 
\begin{pmatrix}
-4 & 1 \\
-1 & 0
\end{pmatrix}, 
\begin{pmatrix}
-4 & 5 \\
3 & 0
\end{pmatrix}, 
\begin{pmatrix}
-4 & -7 \\
7 & 0
\end{pmatrix}, 
\begin{pmatrix}
-4 & -3 \\
-5 & 0
\end{pmatrix}
 \bmod{16}. \notag  
\end{align}
$\mathrm{(4)}$ If 
$
M
\equiv
\pm
\begin{pmatrix}
2 & -1 \\
1 & 2
\end{pmatrix}
\bmod{4}, 
$ 
we have  
\begin{align*}
&\hspace{10mm} g(M)\equiv 0 \bmod{4} 
\Longleftrightarrow \,
\frac{a+d}{4}\equiv -1 \bmod{4} \notag \\
&\Longleftrightarrow
M
\equiv 
\begin{pmatrix}
2 & 3 \\
1 & -6
\end{pmatrix}, 
\begin{pmatrix}
6 & 3 \\
1 & 6
\end{pmatrix}, 
\begin{pmatrix}
-6 & 3 \\
1 & 2
\end{pmatrix}, 
\begin{pmatrix}
-2 & 3 \\
1 & -2
\end{pmatrix}, 
\begin{pmatrix}
2 & 7 \\
5 & -6
\end{pmatrix}, 
\begin{pmatrix}
6 & 7 \\
5 & 6
\end{pmatrix}, 
\begin{pmatrix}
-6 & 7 \\
5 & 2
\end{pmatrix}, 
\begin{pmatrix}
-2 & 7 \\
5 & -2
\end{pmatrix},   \\
&\hspace{8mm}
\begin{pmatrix}
2 & -5 \\
-7 & -6
\end{pmatrix}, 
\begin{pmatrix}
6 & -5 \\
-7 & 6
\end{pmatrix}, 
\begin{pmatrix}
-6 & -5 \\
-7 & 2
\end{pmatrix}, 
\begin{pmatrix}
-2 & -5 \\
-7 & -2
\end{pmatrix}, 
\begin{pmatrix}
2 & -1 \\
-3 & -6
\end{pmatrix}, 
\begin{pmatrix}
6 & -1 \\
-3 & 6
\end{pmatrix}, 
\begin{pmatrix}
-6 & -1 \\
-3 & 2
\end{pmatrix}, 
\begin{pmatrix}
-2 & -1 \\
-3 & -2
\end{pmatrix},  \\
&\hspace{8mm}
\begin{pmatrix}
2 & -7 \\
-5 & -6
\end{pmatrix}, 
\begin{pmatrix}
6 & -7 \\
-5 & 6
\end{pmatrix}, 
\begin{pmatrix}
-6 & -7 \\
-5 & 2
\end{pmatrix}, 
\begin{pmatrix}
-2 & -7 \\
-5 & -2
\end{pmatrix},    
\begin{pmatrix}
2 & -3 \\
-1 & -6
\end{pmatrix}, 
\begin{pmatrix}
6 & -3 \\
-1 & 6
\end{pmatrix}, 
\begin{pmatrix}
-6 & -3 \\
-1 & 2
\end{pmatrix}, 
\begin{pmatrix}
-2 & -3 \\
-1 & -2
\end{pmatrix},   \\
&\hspace{8mm}
\begin{pmatrix}
2 & 1 \\
3 & -6
\end{pmatrix}, 
\begin{pmatrix}
6 & 1 \\
3 & 6
\end{pmatrix}, 
\begin{pmatrix}
-6 & 1 \\
3 & 2
\end{pmatrix}, 
\begin{pmatrix}
-2 & 1 \\
3 & -2
\end{pmatrix}, 
\begin{pmatrix}
2 & 5 \\
7 & -6
\end{pmatrix}, 
\begin{pmatrix}
6 & 5 \\
7 & 6
\end{pmatrix}, 
\begin{pmatrix}
-6 & 5 \\
7 & 2
\end{pmatrix}, 
\begin{pmatrix}
-2 & 5 \\
7 & -2
\end{pmatrix}
\bmod{16}. 
\end{align*}
}
\end{lemma}

\begin{proof}
We treat cases (1) and (2). The other cases can be proved in the same way. 
\newline
{\bf Case (1).}
From the definition of $g(M),$ it follows that 
\begin{align*}
g(M)\equiv 0 \bmod{4}
\Longleftrightarrow&\,  
3(d-c-1)+\frac14(b-c)d+\frac14(a-d)c+cd(1-ad)  \equiv 0 \bmod{4}  \\
\Longleftrightarrow& \,
3(d-1)+\frac{b-c}{4}\cdot d\equiv 0 \bmod{4} \\
\Longleftrightarrow& \,
3(d^2-d)+\frac{b-c}{4}\cdot d^2 \equiv 0 \bmod{4} \\
\Longleftrightarrow& \,
\frac{b-c}{4}\equiv -d+1  \bmod{4}. 
\end{align*}
Because $ad-bc=1,$ it follows that $ad\equiv 1 \bmod{16},$ 
which implies that 
\begin{equation*}
(a,d)\equiv 
\pm(1,1), \pm(5,-3), \pm(-7,-7), \pm(-3,5) \bmod{16}. 
\end{equation*}
\par
If 
$
M
\equiv 
\begin{pmatrix}
1 & 0 \\
0 & 1
\end{pmatrix}
\bmod{4}, 
$ we have 
$
\displaystyle
b-c \equiv 0 \bmod{16},
$ 
which implies that 
\begin{equation*}
(b,c)\equiv (0,0), (4,4), (-4,-4), (8,8) \bmod{16}.
\end{equation*}
\par
If 
$
M
\equiv 
-
\begin{pmatrix}
1 & 0 \\
0 & 1
\end{pmatrix}
\bmod{4}, 
$ we have 
$
\displaystyle
b-c \equiv 8 \bmod{16},
$ 
which implies that 
\begin{equation*}
(b,c)\equiv (0,8), (4,-4), (8,0), (-4,4) \bmod{16}.
\end{equation*}
\newline
{\bf Case (2).}
From the definition of $g(M),$ it follows that 
\begin{align*}
g(M)\equiv 0 \bmod{4}
\Longleftrightarrow&\,  
3(d-c-1)+\frac14(b-c)d+\frac14(a-d)c+cd(1-ad)  \equiv 0 \bmod{4}  \\
\Longleftrightarrow& \,
-(d+1)+\frac{b-c}{4}\cdot d  +\frac{a-d}{4}\cdot 2  \equiv 0 \bmod{4} \\
\Longleftrightarrow& \,
-(d^2+d)+\frac{b-c}{4}\cdot d^2  +\frac{a-d}{4}\cdot 2d  \equiv 0 \bmod{4} \\
\Longleftrightarrow& \,
-(1+d)+\frac{b-c}{4}+\frac{a-d}{4}\cdot2\equiv 0 \bmod{4}.
\end{align*}
Because $ad-bc=1,$ it follows that $ad-4\equiv 1 \bmod{8},$ 
which implies that 
\begin{equation*}
(a,d)\equiv 
\pm
(1,-3), 
\pm
(-3,1) \bmod{8} 
\Longrightarrow
a-d\equiv 4 \bmod{8}. 
\end{equation*}
Therefore, we find that 
\begin{equation*}
g(M)\equiv 0 \bmod{4}
\Longleftrightarrow
\frac{b-c}{4} \equiv d-1 \bmod{4}. 
\end{equation*}
\par
If 
$
M
\equiv 
\begin{pmatrix}
1 & 2 \\
2 & 1
\end{pmatrix}
\bmod{4}, 
$ we have 
$
b-c\equiv 0 \bmod{16}, 
$
which implies that 
\begin{equation*}
(b,c)\equiv 
\pm(2,2), 
\pm(6,6) \bmod{16}.
\end{equation*}
Because $ad-bc=1,$ we see that 
$ad-4\equiv 1 \bmod{16},$ 
which implies that 
\begin{equation*}
(a,d)\equiv 
(1,5), (5,1), (-7,-3), (-3,-7) \bmod{16}. 
\end{equation*}
\par
If 
$
M
\equiv 
-
\begin{pmatrix}
1 & 2 \\
2 & 1
\end{pmatrix}
\bmod{4}, 
$ we have 
$
b-c\equiv 8 \bmod{16}, 
$
\begin{equation*}
(b,c)\equiv 
\pm(2,-6), 
\pm(6,-2) \bmod{16}.
\end{equation*}
Because $ad-bc=1,$ we see that 
$ad-4\equiv 1 \bmod{16},$ 
which implies that 
\begin{equation*}
(a,d)\equiv 
(-1,-5), (-5,-1), (7,3), (3,7) \bmod{16}. 
\end{equation*}
\end{proof}

\begin{theorem}
\label{thmk=3-level4}
{\it
Suppose that 
$k\equiv \pm 3 \bmod{12}$. 
Then, it follows that 
\begin{align*}
\Ker \nu_{G^{k}}
=
\Big\{
M=
\begin{pmatrix}
a & b \\
c & d
\end{pmatrix}
\in
\Gamma_{\theta,4} \,\Big| \, 
&  &\frac{b-c}{4} \equiv& -d+1  & &\bmod{4}  &  &(b\equiv c \equiv  0\bmod{4})   \\
&  &\frac{b-c}{4} \equiv& \, d-1 & &\bmod{4}  &  &(b\equiv c \equiv  2\bmod{4})   \\
&  &\frac{a+d}{4}\equiv& -1 &  &\bmod{4} & &(b\equiv c\equiv 1 \bmod{2})   
\Big\}.
\end{align*}
Moreover, 
as a coset decomposition of $\Gamma_{\theta,4}$ modulo $\Ker \nu_{G^{k}}$,  
we may choose $S^{n}$ $(n=0,4,8,12),$ where 
$
S=
\begin{pmatrix}
1 & 1 \\
0 & 1
\end{pmatrix}. 
$
}
\end{theorem}

\begin{proof}
For each $M\in\Gamma_{\theta,4},$ 
we have 
\begin{equation*}
\nu_{G^k}(M)=
\exp
\left\{
\pm
\frac{\pi i}{2}
g(M)
\right\},
\end{equation*}
which proves the theorem. 
\par
The coset decomposition can be obtained by considering the values of $\nu_{G^{k}  }(M).$
\end{proof}

\subsection{Case where $k\equiv \pm 4 \bmod{12}$}

\begin{lemma}
\label{lem:level4-mod3}
{\it
Set 
$
M=
\begin{pmatrix}
a & b \\
c & d
\end{pmatrix}
\in\Gamma_{\theta,4}.
$
\newline
$\mathrm{(1)}$ If 
$
M
\equiv
\pm
\begin{pmatrix}
1 & 0 \\
0 & 1
\end{pmatrix}, 
\pm
\begin{pmatrix}
1 & 2 \\
2 & 1
\end{pmatrix}
\bmod{4},
$ 
we have 
\begin{align*}
g(M)\equiv 0 \bmod{3} 
\Longleftrightarrow
M
\equiv& 
\pm
\begin{pmatrix}
1 & 0 \\
0 & 1
\end{pmatrix}, 
\pm
\begin{pmatrix}
0 & -1 \\
1 & 0
\end{pmatrix}, 
\pm
\begin{pmatrix}
1 & 1 \\
1 & -1
\end{pmatrix},  
\pm
\begin{pmatrix}
1 & -1 \\
-1 & -1
\end{pmatrix}
\bmod{3}.   
\end{align*}
\newline
$\mathrm{(2)}$ If 
$
M
\equiv
\pm
\begin{pmatrix}
0 & -1 \\
1 & 0
\end{pmatrix}, 
\pm
\begin{pmatrix}
2 & -1 \\
1 & 2
\end{pmatrix}
\bmod{4},
$ 
we have 
\begin{align*}
g(M)\equiv 0 \bmod{3} 
\Longleftrightarrow
M
\equiv& 
\pm
\begin{pmatrix}
1 & 0 \\
0 & 1
\end{pmatrix}, 
\pm
\begin{pmatrix}
0 & -1 \\
1 & 0
\end{pmatrix}, 
\pm
\begin{pmatrix}
1 & 1 \\
1 & -1
\end{pmatrix},  
\pm
\begin{pmatrix}
1 & -1 \\
-1 & -1
\end{pmatrix}
\bmod{3}.   
\end{align*}
}
\end{lemma}

\begin{proof}
We treat case (1). The other case can be proved in the same way. 
From the definition of $g(M),$ it follows that 
\begin{align*}
g(M)\equiv 0 \bmod{3}
\Longleftrightarrow&
3(d-c-1)+\frac14(b-c)d+\frac14(a-d)c+cd(1-ad)  \equiv 0 \bmod{3} \\
\Longleftrightarrow&
(b-c)d+(a-d)c+cd(1-ad) \equiv 0 \bmod{3}.
\end{align*}
\par
If $d\equiv 0\bmod{3},$ 
we see that  
\begin{equation*}
g(M)\equiv 0 \bmod{3} 
\Longleftrightarrow
ac\equiv 0 \bmod{3}.
\end{equation*}
Because $ad-bc=1,$ 
it follows that 
$
c 
\not\equiv 0 \bmod{3}, 
$ 
which implies that $a\equiv 0 \bmod{3}.$ 
As $ad-bc=1,$ 
we find that 
$bc\equiv -1 \bmod{3},$ 
which implies that 
\begin{equation*}
(a,d)\equiv(0,0) \bmod{3}, 
(b,c)\equiv \pm (-1,1) \bmod{3}
\Longleftrightarrow
M
\equiv 
\pm
\begin{pmatrix}
0 & -1 \\
1 & 0
\end{pmatrix}
\bmod{3}.
\end{equation*}
\par
If $d \not\equiv 0 \bmod{3},$ 
we see that 
\begin{align*}
g(M)\equiv 0 \bmod{3}
\Longleftrightarrow&
(b-c)d+(a-d)c+cd(1-ad) \equiv 0 \bmod{3} \\
\Longleftrightarrow&
(b-c)d
+
ac-cd+cd-acd^2 \equiv 0 \bmod{3}  \\
\Longleftrightarrow&
(b-c)d\equiv 0 \bmod{3} \\
\Longleftrightarrow&
b-c\equiv 0 \bmod{3} \\
\Longleftrightarrow&
(b,c)\equiv (0,0), (1,1), (-1,-1) \bmod{3} \\
\Longleftrightarrow&
M
\equiv 
\pm
\begin{pmatrix}
1 & 0 \\
0 & 1
\end{pmatrix}, 
\pm
\begin{pmatrix}
1 & 1 \\
1 & -1
\end{pmatrix}, 
\pm
\begin{pmatrix}
-1 & -1 \\
-1 & 1
\end{pmatrix} 
\bmod{3}. 
\end{align*}
\end{proof}

\begin{theorem}
\label{thmk=4-level4}
{\it
Suppose that 
$k\equiv \pm 4 \bmod{12}$. 
Then, it follows that 
\begin{align*}
\Ker \nu_{G^{k}}
=&
\Big\{
M
\in
\Gamma_{\theta, 4} \,\Big| \, 
M
\equiv& 
\pm
\begin{pmatrix}
1 & 0 \\
0 & 1
\end{pmatrix}, 
\pm
\begin{pmatrix}
0 & -1 \\
1 & 0
\end{pmatrix}, 
\pm
\begin{pmatrix}
1 & 1 \\
1 & -1
\end{pmatrix},  
\pm
\begin{pmatrix}
1 & -1 \\
-1 & -1
\end{pmatrix}
\bmod{3}
\Big\}.
\end{align*}
Moreover, 
as a coset decomposition of $\Gamma_{\theta,4}$ modulo $\Ker \nu_{G^{k}}$,  
we may choose $S^{n}$ $(n=0,4,8),$ where 
$
S=
\begin{pmatrix}
1 & 1 \\
0 & 1
\end{pmatrix}. 
$
}
\end{theorem}

\begin{proof}
For each $M\in\Gamma_{\theta,4},$ 
we have 
\begin{equation*}
\nu_{G^k}(M)
=
\exp
\left\{
\pm
\frac{2\pi i}{3}
g(M)
\right\}. 
\end{equation*}
The theorem follows from Lemma \ref{lem:level4-mod3}.  
\par
The coset decomposition can be obtained by considering the values of $\nu_{G^{k}  }(M).$
\end{proof}

\subsection{Case where $k\equiv \pm 2 \bmod{12}$}

\begin{theorem}
\label{thmk=2-level4}
{\it
Suppose that 
$k\equiv \pm2 \bmod{12}$. 
Then, it follows that 
\begin{align*}
\Ker \nu_{G^{k}}
=
\Ker 
\nu_{ G^{6} }
\cap
\Ker \nu_{ G^4 }. 
\end{align*}
Moreover, 
as a coset decomposition of $\Gamma_{\theta,4}$ modulo $\Ker \nu_{G^{k}}$,  
we may choose 
\begin{equation*}
S^{n}  \,\, 
(n=0,4,8,12,16,20),  \,\,
S=
\begin{pmatrix}
1 & 1 \\
0 & 1
\end{pmatrix}. 
\end{equation*}
}
\end{theorem}

\begin{proof}
For each 
$
M
=
\begin{pmatrix}
a & b \\
c & d
\end{pmatrix}
\in
\Gamma_{\theta,4}, 
$ 
we have 
\begin{equation*}
\nu_{G^{k}  }(M)=
\exp
\left\{
\pm
\frac{
 \pi i
}
{3} 
g(M) 
\right\},
\end{equation*}
which implies that 
\begin{equation*}
M
\in
\Ker \nu_{G^{k}}
\Longleftrightarrow 
g(M)\equiv 0 \bmod{6}
\Longleftrightarrow 
g(M)\equiv 0 \bmod{2} \,\,
\mathrm{and} \,\,
g(M)\equiv 0 \bmod{3}. 
\end{equation*}
The theorem follows from Theorems \ref{thmk=6-level4} and \ref{thmk=4-level4}.
\par
The coset decomposition can be obtained by considering the values of $\nu_{G^{k}  }(M).$
\end{proof}

\subsection{Case where $k\equiv \pm 1, \pm 5 \bmod{12}$}

\begin{theorem}
\label{thmk=1and5-level4}
{\it
Suppose that 
$k\equiv \pm1, \pm 5 \bmod{12}$. 
Then, it follows that 
\begin{align*}
\Ker \nu_{G^{k}}
=
\Ker 
\nu_{ G^{3} }
\cap
\Ker \nu_{ G^4 }. 
\end{align*}
Moreover, 
as a coset decomposition of $\Gamma_{\theta,4}$ modulo $\Ker \nu_{G^{k}}$,  
we may choose 
\begin{equation*}
S^{n}  \,\, 
(n=0,4,8,12,16,20,24,28,32,36,40,44),  \,\,
S=
\begin{pmatrix}
1 & 1 \\
0 & 1
\end{pmatrix}. 
\end{equation*}
}
\end{theorem}

\begin{proof}
For each 
$
M
=
\begin{pmatrix}
a & b \\
c & d
\end{pmatrix}
\in
\Gamma_{\theta,4}, 
$ 
we have 
\begin{equation*}
\nu_{ G^{k}  }(M)=
\exp
\left\{
\pm
\frac{
k
 \pi i
}
{6} 
g(M) 
\right\}, \,\,
(k,6)=1, 
\end{equation*}
which implies that 
\begin{equation*}
M
\in
\Ker \nu_{G^{k}}
\Longleftrightarrow 
g(M)\equiv 0 \bmod{12}
\Longleftrightarrow 
g(M)\equiv 0 \bmod{4} \,\,
\mathrm{and} \,\,
g(M)\equiv 0 \bmod{3}. 
\end{equation*}
The theorem follows from Theorems \ref{thmk=3-level4} and \ref{thmk=4-level4}.
\par
The coset decomposition can be obtained by considering the values of $\nu_{ G^{k}  }(M).$
\end{proof}

\section{Coset decompositions of $\Gamma(1) $ modulo $\Gamma_{\theta,3}$ and $\Gamma_{\theta,4}$  }
\label{sec:coset}

For
$
M
=
\begin{pmatrix}
a & b \\
c & d
\end{pmatrix}
\in
\Gamma(1),
$ 
we define an element $\lambda_N(M)$ of $SL(2, \mathbb{Z}/N \mathbb{Z} )$ by 
\begin{equation*}
\lambda_N(M)
=
\begin{pmatrix}
\bar{a} & \bar{b} \\
\bar{c} & \bar{d}
\end{pmatrix},
\end{equation*}
where 
$\bar{a} \equiv a \bmod{N}, $ 
$\bar{b} \equiv b \bmod{N}, $
$\bar{c} \equiv c \bmod{N}, $ 
and 
$\bar{d} \equiv d \bmod{N}. $
In addition, 
we define 
\begin{equation*}
\Gamma(N)
=
\left\{
M
=
\begin{pmatrix}
a & b \\
c & d
\end{pmatrix}
\in
\Gamma(1) \, | \,
M\equiv 
 \begin{pmatrix}
1 & 0 \\
0 & 1
\end{pmatrix} 
\bmod{N}
\right\}.
\end{equation*}

\subsection{Coset decomposition of $\Gamma(1) $ modulo $\Gamma_{\theta,3}$ }

\begin{theorem}
\label{them:parabolic-level3}
{\it
The parabolic points of $\Gamma_{\theta,3}$ are given by $\infty$ and $-1.$ 
}
\end{theorem}

\begin{proof}
Considering $\lambda_3$, 
as a  coset decomposition of $\Gamma(1)$ modulo $\Gamma(3),$ 
we may choose 
\begin{align*}
&
\pm 
\begin{pmatrix}
1 & 0 \\
0 & 1
\end{pmatrix}, 
\pm 
\begin{pmatrix}
1 & 1 \\
0 & 1
\end{pmatrix}, 
\pm 
\begin{pmatrix}
1 & -1 \\
0 & 1
\end{pmatrix}, 
\pm 
\begin{pmatrix}
1 & -1 \\
1 &0
\end{pmatrix}, 
\pm 
\begin{pmatrix}
1 & 0 \\
1 & 1
\end{pmatrix}, 
\pm 
\begin{pmatrix}
0 & -1 \\
1 & 0
\end{pmatrix},   \\
&
\pm 
\begin{pmatrix}
0 & -1 \\
1 & 1
\end{pmatrix}, 
\pm 
\begin{pmatrix}
0 & 1 \\
-1 & 1
\end{pmatrix}, 
\pm 
\begin{pmatrix}
1 & 1 \\
-1 & 0
\end{pmatrix}, 
\pm 
\begin{pmatrix}
1 & 0 \\
-1 & 1
\end{pmatrix}, 
\pm 
\begin{pmatrix}
1 & 1 \\
1 & 2
\end{pmatrix}, 
\pm 
\begin{pmatrix}
1 & -1 \\
-1 & 2
\end{pmatrix},
\end{align*}
which implies that 
\begin{align*}
\Gamma(1)=&
\Gamma_{\theta,3} 
\cup
\Gamma_{\theta,3} 
\begin{pmatrix}
1 & 1 \\
0 & 1
\end{pmatrix} 
\cup
\Gamma_{\theta,3} 
\begin{pmatrix}
1 & 2 \\
0 & 1
\end{pmatrix} 
\cup
\Gamma_{\theta,3} 
\begin{pmatrix}
1 & 0 \\
-1 & 1
\end{pmatrix} 
\cup
\Gamma_{\theta,3} 
\begin{pmatrix}
1 & 1 \\
-1& 0
\end{pmatrix} 
\cup
\Gamma_{\theta,3} 
\begin{pmatrix}
-1 & 1 \\
1 & -2
\end{pmatrix} \\
=&
\Gamma_{\theta,3} 
\cup
\Gamma_{\theta,3} S
\cup
\Gamma_{\theta,3} S^{-1}
\cup
\Gamma_{\theta,3} ST^{-1}
\cup
\Gamma_{\theta,3} S^{-1} T^{-1}
\cup
\Gamma_{\theta,3} S^{-1} T^{-1} S,
\end{align*}
where 
$
S=\begin{pmatrix}
1 & 1 \\
0 & 1
\end{pmatrix} 
$ 
and 
$
T=
\begin{pmatrix}
0 & -1 \\
1 & 0
\end{pmatrix}.  
$
\end{proof}

\begin{theorem}
{\it
An elliptic element of $\Gamma_{\theta,3}$ is equivalent to 
\begin{equation*}
\pm 
\begin{pmatrix}
0 & -1 \\
1 & 0
\end{pmatrix}, 
\pm 
\begin{pmatrix}
-3 & -2 \\
5 & 3
\end{pmatrix}.
\end{equation*}
}
\end{theorem}

\begin{proof}
Recall that 
an elliptic element of $\Gamma_{\theta,3}$ is conjugate to 
\begin{equation*}
\pm
M
\begin{pmatrix}
0 & -1 \\
1 & 0
\end{pmatrix}
M^{-1}, 
\pm
M
\begin{pmatrix}
0 & 1 \\
-1 & -1
\end{pmatrix}
M^{-1}, 
\pm
M
\begin{pmatrix}
-1 & -1 \\
1 & 0
\end{pmatrix}
M^{-1},
\end{equation*}
where $M$ is a coset representative in the proof of Theorem \ref{them:parabolic-level3}. 
We obtain 
\begin{equation*}
\pm 
\begin{pmatrix}
0 & -1 \\
1 & 0
\end{pmatrix}, 
\pm 
\begin{pmatrix}
-3 & -2 \\
5 & 3
\end{pmatrix}.
\end{equation*}
\par
Suppose that the matrices of the theorem are equivalent, which implies that 
there exists 
$
\begin{pmatrix}
a & b \\
c & d
\end{pmatrix}
\in
\Gamma_{\theta,3}
$ 
such that 
\begin{equation*}
\pm
\begin{pmatrix}
a & b \\
c & d
\end{pmatrix}
\begin{pmatrix}
0 & -1 \\
1 & 0
\end{pmatrix}
=
\begin{pmatrix}
-3 & -2 \\
5 & 3
\end{pmatrix}
\begin{pmatrix}
a & b \\
c & d
\end{pmatrix}. 
\end{equation*}
\par
If 
$
\begin{pmatrix}
a & b \\
c & d
\end{pmatrix}
\begin{pmatrix}
0 & -1 \\
1 & 0
\end{pmatrix}
=
\begin{pmatrix}
-3 & -2 \\
5 & 3
\end{pmatrix}
\begin{pmatrix}
a & b \\
c & d
\end{pmatrix}, 
$ 
we have 
\begin{equation*}
-3a-b\equiv a-3b \equiv 0 \bmod{2},
\end{equation*}
and
\begin{equation*}
c=-\frac{a+3b}{2}, d=\frac{a-3b}{2}.
\end{equation*}
Since $a\equiv d \bmod{3}$ and $b\equiv-c \bmod{3},$ 
it follows that 
$a\equiv b\equiv 0 \bmod{3},$ which is impossible. 
\par
If 
$
-
\begin{pmatrix}
a & b \\
c & d
\end{pmatrix}
\begin{pmatrix}
0 & -1 \\
1 & 0
\end{pmatrix}
=
\begin{pmatrix}
-3 & -2 \\
5 & 3
\end{pmatrix}
\begin{pmatrix}
a & b \\
c & d
\end{pmatrix}, 
$ 
we have 
\begin{equation*}
-3a+b\equiv a+3b \equiv 0 \bmod{2},
\end{equation*}
and
\begin{equation*}
c=-\frac{3a-b}{2}, d=-\frac{a+3b}{2}.
\end{equation*}
Since $ad-bc=1,$ 
it follows that 
\begin{equation*}
a^2+b^2=-2,
\end{equation*}
which is impossible. 
\end{proof}

\subsection{Coset decomposition of $\Gamma(1) $ modulo $\Gamma_{\theta,4}$ }

\begin{theorem}
\label{them:parabolic-level4}
{\it
The parabolic points of $\Gamma_{\theta,4}$ are given by $\infty$ and $-1.$ 
}
\end{theorem}

\begin{proof}
Considering $\lambda_4$, 
as a  coset decomposition of $\Gamma(1)$ modulo $\Gamma(4),$ 
we may choose 
\begin{align*}
&
\pm 
\begin{pmatrix}
1 & 0 \\
0 & 1
\end{pmatrix}, 
\pm 
\begin{pmatrix}
1 & 0 \\
2 & 1
\end{pmatrix}, 
\pm 
\begin{pmatrix}
1 & 2 \\
0 & 1
\end{pmatrix}, 
\pm 
\begin{pmatrix}
1 & 2 \\
2&5
\end{pmatrix}, 
\pm 
\begin{pmatrix}
0 & -1 \\
1 & 0
\end{pmatrix}, 
\pm 
\begin{pmatrix}
0 & -1 \\
1 & 2
\end{pmatrix},   \\
&
\pm 
\begin{pmatrix}
2 & -1 \\
1 & 0
\end{pmatrix}, 
\pm 
\begin{pmatrix}
2 & 3 \\
1 & 2
\end{pmatrix}, 
\pm 
\begin{pmatrix}
1 & 1 \\
0 & 1
\end{pmatrix}, 
\pm 
\begin{pmatrix}
1 & -1 \\
0 & 1
\end{pmatrix}, 
\pm 
\begin{pmatrix}
1 & 1 \\
2 & 3
\end{pmatrix}, 
\pm 
\begin{pmatrix}
1 & -1 \\
2 & -1
\end{pmatrix},                \\
&
\pm 
\begin{pmatrix}
2 & 1 \\
1 & 1
\end{pmatrix}, 
\pm 
\begin{pmatrix}
-2 & 1 \\
1 & -1
\end{pmatrix}, 
\pm 
\begin{pmatrix}
0 & -1 \\
1 & -1
\end{pmatrix}, 
\pm 
\begin{pmatrix}
0 & -1 \\
1 & 1
\end{pmatrix}, 
\pm 
\begin{pmatrix}
1 & 1 \\
1 & 2
\end{pmatrix}, 
\pm 
\begin{pmatrix}
-1 & 1 \\
1 & -2
\end{pmatrix},    \\
&
\pm 
\begin{pmatrix}
-1 & -1 \\
1 & 0
\end{pmatrix}, 
\pm 
\begin{pmatrix}
1 & -1 \\
1 & 0
\end{pmatrix}, 
\pm 
\begin{pmatrix}
1 & 0 \\
1 & 1
\end{pmatrix}, 
\pm 
\begin{pmatrix}
-1 & 0 \\
1 & -1
\end{pmatrix}, 
\pm 
\begin{pmatrix}
1 & -2 \\
1 & -1
\end{pmatrix}, 
\pm 
\begin{pmatrix}
-1 & -2 \\
1 & 1
\end{pmatrix},
\end{align*}
which implies that 
\begin{align*}
\Gamma(1)=&
\Gamma_{\theta,4} 
\cup
\Gamma_{\theta,4} 
\begin{pmatrix}
1 & 1 \\
0 & 1
\end{pmatrix} 
\cup
\Gamma_{\theta,4} 
\begin{pmatrix}
1 & 2 \\
0 & 1
\end{pmatrix} 
\cup
\Gamma_{\theta,4} 
\begin{pmatrix}
1 & -1 \\
0 & 1
\end{pmatrix} 
\cup
\Gamma_{\theta,4} 
\begin{pmatrix}
-1 & 0 \\
1& -1
\end{pmatrix} 
\cup
\Gamma_{\theta,4} 
\begin{pmatrix}
-1 & -1 \\
1 & 0
\end{pmatrix} \\
=&
\Gamma_{\theta,4} 
\cup
\Gamma_{\theta,4} S
\cup
\Gamma_{\theta,4} S^{2}
\cup
\Gamma_{\theta,4} S^{-1}
\cup
\Gamma_{\theta,4} S T^{-1}
\cup
\Gamma_{\theta,4} S^{-1} T^{-1}, 
\end{align*}
where 
$
S=\begin{pmatrix}
1 & 1 \\
0 & 1
\end{pmatrix} 
$ 
and 
$
T=
\begin{pmatrix}
0 & -1 \\
1 & 0
\end{pmatrix}.  
$
\end{proof}

\begin{theorem}
{\it
An elliptic element of $\Gamma_{\theta,4}$ is conjugate to 
\begin{equation*}
\pm 
\begin{pmatrix}
0 & -1 \\
1 & 0
\end{pmatrix}, 
\pm 
\begin{pmatrix}
2 & -5 \\
1 & -2
\end{pmatrix}.
\end{equation*}
}
\end{theorem}

\begin{proof}
Recall that 
an elliptic element of $\Gamma_{\theta,4}$ is conjugate to 
\begin{equation*}
\pm
M
\begin{pmatrix}
0 & -1 \\
1 & 0
\end{pmatrix}
M^{-1}, 
\pm
M
\begin{pmatrix}
0 & 1 \\
-1 & -1
\end{pmatrix}
M^{-1}, 
\pm
M
\begin{pmatrix}
-1 & -1 \\
1 & 0
\end{pmatrix}
M^{-1},
\end{equation*}
where $M$ is a coset representative in the proof of Theorem \ref{them:parabolic-level4}. 
We obtain 
\begin{equation*}
\pm 
\begin{pmatrix}
0 & -1 \\
1 & 0
\end{pmatrix}, 
\pm 
\begin{pmatrix}
2 & -5 \\
1 & -2
\end{pmatrix}.
\end{equation*}
\par
Suppose that the matrices of the theorem are equivalent, which implies that 
there exists 
$
\begin{pmatrix}
a & b \\
c & d
\end{pmatrix}
\in
\Gamma_{\theta,4}
$ 
such that 
\begin{equation*}
\pm
\begin{pmatrix}
a & b \\
c & d
\end{pmatrix}
\begin{pmatrix}
0 & -1 \\
1 & 0
\end{pmatrix}
=
\begin{pmatrix}
2 & -5 \\
1 & -2
\end{pmatrix}
\begin{pmatrix}
a & b \\
c & d
\end{pmatrix}. 
\end{equation*}
\par
If 
$
\begin{pmatrix}
a & b \\
c & d
\end{pmatrix}
\begin{pmatrix}
0 & -1 \\
1 & 0
\end{pmatrix}
=
\begin{pmatrix}
2 & -5 \\
1 & -2
\end{pmatrix}
\begin{pmatrix}
a & b \\
c & d
\end{pmatrix}, 
$ 
we have 
\begin{equation*}
2a-b\equiv a+2b \equiv 0 \bmod{5},
\end{equation*}
and
\begin{equation*}
c=\frac{2a-b}{5}, d=\frac{a+2b}{5}.
\end{equation*}
Since $a\equiv d \bmod{4}$ and $b\equiv-c \bmod{4},$ 
it follows that 
$a\equiv b\equiv 0 \bmod{2},$ which is impossible. 
\par
If 
$
-
\begin{pmatrix}
a & b \\
c & d
\end{pmatrix}
\begin{pmatrix}
0 & -1 \\
1 & 0
\end{pmatrix}
=
\begin{pmatrix}
2 & -5 \\
1 & -2
\end{pmatrix}
\begin{pmatrix}
a & b \\
c & d
\end{pmatrix}, 
$ 
we have 
\begin{equation*}
2a+b\equiv -a+2b \equiv 0 \bmod{5},
\end{equation*}
and
\begin{equation*}
c=\frac{2a+b}{5}, d=\frac{-a+2b}{5}.
\end{equation*}
Since $ad-bc=1,$ 
it follows that 
\begin{equation*}
a^2+b^2=-5,
\end{equation*}
which is impossible. 
\end{proof}




\end{document}